\documentclass[a4paper,11pt]{amsart}
\input{Macros}
\usepackage[frenchb, english]{babel}
\usepackage{tabu}

\newcommand{\fb}{\mathfrak{b}}
\newcommand{\fm}{\mathfrak{m}}

\newcommand{\ff}{\mathfrak{f}}

\DeclareMathOperator{\Sch}{Sch}

\DeclareMathOperator{\Mum}{\mathcal{M}um}
\DeclareMathOperator{\Crys}{Crys}
\DeclareMathOperator{\RHom}{RHom}

\title{Pathological MMP singularities as $\alpha_p$-quotients}
\author{Quentin Posva}
\date{}

\address{Mathematisches Institut der Heinriche-Heine-Universität Düsseldorf, Universitätsstr.1, 24.13.03.33}
\email{quentin.posva@hhu.de}

\begin{document}

\maketitle

\begin{quote}
\textsc{Abstract.} We construct pathological examples of MMP singularities in every positive characteristic using quotients by $\alpha_p$-actions. In particular, we obtain non-$S_3$ terminal singularities, as well as locally stable (respectively stable) families whose general fibers are smooth (respectively klt, Cohen--Macaulay and $F$-injective) and whose special fibers are non-$S_2$. The dimensions of these examples are bounded below by a linear function of the characteristic.
\end{quote}

\tableofcontents

\section{Introduction}
Singularities of the Minimal Model Program (MMP) have poor cohomological properties in positive characteristic, unlike in characteristic $0$. Striking examples are canonical and terminal singularities that are not Cohen--Macaulay \cite{Kovacs_Non_CM_canonical_sing,
Bernasconi_Non_normal_plt_centers_in_pos_char, Yasuda_Discrepancies_p_cyclic_quotients, Totaro_Failure_Kodaira_vanishing_for_Fanos, Totaro_Terminal_non_CM_3folds}. In this article, we construct several new examples of non-$S_3$ singularities of the MMP in every positive characteristic $p>0$. 

The examples constructed in the aforementioned works are of two types: their stem either from a cone construction on a variety failing Kodaira vanishing, or from a wild $\bZ/p\bZ$-quotient. Our examples are $\alpha_p$-quotients: to understand their singularities we use the perspective of $1$-foliations and the tools developed in \cite{Posva_Singularities_quotients_by_1_foliations}.

Our first series of examples gives the following:

\begin{theorem_intro}[\autoref{thm:non_S3_qt}]\label{thm_intro:terminal_non_S3}
Let $k$ be an algebraically closed field of characteristic $p>0$. Then there exists an isolated $\bQ$-factorial singularity $(\bold{0}\in Y)$ over $k$ that is canonical (resp. terminal) of dimension $\max\{p, 3\}$ (resp. $p+1$) and non-$S_3$.
\end{theorem_intro}

In particular, we obtain a three-dimensional terminal singularity in characteristic $p=2$ that is not $S_3$: the other knowns examples of this dimension were all constructed by Totaro in positive and mixed characteristics $2,3$ and $5$ \cite{Totaro_Failure_Kodaira_vanishing_for_Fanos, Totaro_Terminal_non_CM_3folds}. For an arbitrary characteristic $p>2$, Totaro also constructed terminal non-$S_3$ cone singularities of dimension $2p+2$ \cite{Totaro_Failure_Kodaira_vanishing_for_Fanos}. In this regard, we achieve an improvement on the asymptotic dimensional lower bound. 

Note that Yasuda exhibits in \cite{Yasuda_Discrepancies_p_cyclic_quotients} a series of terminal and canonical singularities, starting in dimension $4$, which have depth $\geq 3$ but are not Cohen--Macaulay: their optimal asymptotic dimension, as a function of $p$, is roughly given by $\sqrt{2p}$.

We establish several other properties of the singularities $(\bold{0}\in Y)$ in \autoref{prop:not_1_Gor}, \autoref{prop:CM_up_to_nilpotence}, \autoref{cor:not_F_inj} and \autoref{prop:rational_up_to_Frob}. Notably, they are not $F$-injective, and they are rational up to Frobenius nilpotence.

\medskip
We use the same method to produce another series of pathological examples. In characteristic $0$, fibers of locally stable one-parameter families (in the sense of \autoref{def:local_stability}) are $S_2$ \cite[\S 7.3]{Kollar_Singularities_of_the_minimal_model_program}. This is true for one-parameter families of surfaces with residue characteristics different from $2$, $3$ and $5$ \cite{Arvidsson_Bernasconi_Patakfalvi_Properness_moduli_surfaces}, but recent examples of Koll\'{a}r \cite{Kollar_Families_of_3folds_in_pos_char} show that this property can fail dramatically for families of threefolds in every positive characteristic. Using $\alpha_p$-quotients, we construct new examples of this kind. First we obtain:

\begin{theorem_intro}[\autoref{thm:loc_stable_families_I}]\label{thm_intro:local_stability}
Let $k$ be an algebraically closed field of characteristic $p>0$. Then there exists a locally stable family $Y\to \bA^1$ over $k$ of relative dimension $\max\{p,3\}$ such that $Y_t$ is smooth for $t\neq 0$ and $Y_0$ is reduced but non-$S_2$.
\end{theorem_intro}



Then we compactify the families of \autoref{thm_intro:local_stability} to obtain pathological stable families (in the sense of \autoref{def:local_stability}):

\begin{theorem_intro}[\autoref{thm:stable_families}]\label{thm_intro:stability}
Let $k$ be an algebraically closed field of characteristic $p>0$. Then there exists a projective flat family of pairs $(\sY,\sB)\to \bA^1$ over $k$ of relative dimension $\max\{p,3\}$ such that:
	\begin{enumerate}
		\item\label{item_intro:a} $\sY$ is normal and $(\sY,\sB)\to \bA^1$ is stable,
		\item\label{item_intro:b} $\sY_t$ has only $\mu_p$-quotient singularities (or is regular if $p=2$) for $t\neq 0$,
		\item\label{item_intro:c} $\sY_0$ is reduced but non-$S_2$.
	\end{enumerate}
Moreover, if $C$ is a normal curve and if $C\to \bA^1$ is a finite flat morphism, then $\sY_C$ is normal and $(\sY_C,\sB_C)\to C$ is stable.
\end{theorem_intro}

Our method produces boundary divisors $\sB\neq 0$ whose coefficients are quite small, but the number of irreducible components of $\sB$ and the volume of $K_{\sY_t}+\sB_t$ are up to our liking: see \autoref{rmk:boundary_divisor} and the proof of \autoref{thm:KSBA_mod_stacks}.

The examples given by \autoref{thm_intro:stability} differ from Koll\'{a}r's ones \cite{Kollar_Families_of_3folds_in_pos_char} in two important ways: their dimension is constrained by the characteristic (while Koll\'{a}r constructs pathological families of threefold pairs in every characteristic, and indicates how to generalize the construction to any higher dimension), but their general fibers are klt, Cohen--Macaulay and $F$-injective by \cite[Theorems 1 and 2]{Posva_Singularities_quotients_by_1_foliations} (while the general fibres in Koll\'{a}r's examples are not klt nor Cohen--Macaulay: see \autoref{rmk:sing_Kollar_example}). Boundary divisors appear in both situations. We discuss the implications of these examples for moduli theory in \autoref{section:implications_moduli_theory} below.

\medskip
Let me sketch the construction of these examples. We start with a $1$-foliation $\sF$ of rank $1$ on an affine space $\bA^n$, and we wish to understand the singularities of $\bA^n/\sF$. There is a simple criterion guaranteeing that $\bA^n/\sF$ is only $S_2$ (see \cite[Remark 2.5.4]{Posva_Singularities_quotients_by_1_foliations}), so we only need to study its singularities in the sense of the MMP. If $F$ is a divisor over $\bA^n/\sF$, then by \cite[\S 4.2]{Posva_Singularities_quotients_by_1_foliations} we can compute the discrepancy $a(F;\bA^n/\sF)$ in terms of $a(E;\bA^n)$ and $a(E;\sF)$ where $E$ is the closure of $F$ in the function field of $\bA^n$ (this is recalled in \autoref{thm:sing_qt}). The problem is that we cannot estimate $a(E;\sF)$ easily. So we blow-up $\bA^n$ until the singularities of $\sF$ become log canonical. This leads to a commutative diagram
		$$\begin{tikzcd}
		X\arrow[d, "b"] \arrow[r] & Z=X/b^*\sF \arrow[d, "\fb"] \\
		\bA^n \arrow[r] & \bA^n/\sF
		\end{tikzcd}$$
where the horizontal arrows are the quotient morphisms, and  both morphisms $b$ and $\fb$ are birational. Writing 
		$$K_Z-\sum_ia_i F_i=\fb^*K_{\bA^n/\sF}$$
where the $F_i$'s are the $\fb$-exceptional divisors, we reduce to computing the discrepancies over the sub-pair $(Z, -\sum_ia_i F_i)$. The advantage is now that $b^*\sF$ is log canonical, so the tools given by \autoref{thm:sing_qt} are much more efficient. The appearance of the characteristic in the computations is due to the adjunction formula of \autoref{prop:log_adjunction}. 

For \autoref{thm_intro:local_stability} and \autoref{thm_intro:stability} we use a relative version of this construction. Proving that the central fiber of the quotient is not $S_2$ is more involved (as the criterion used before does not apply): we use the tools developed in \cite[\S 5]{Posva_Singularities_quotients_by_1_foliations} (see \autoref{thm:family_of_quotients}). To ensure that the fibers of the quotient family are regular in codimension one, we need to work in relative dimension $d$ at least $3$. When $d=2=p$ the output (\autoref{example:extreme_case_I}) is quite different: we obtain a locally stable family of surfaces $\sS\to \bA^1$ where $\sS$ is Cohen--Macaulay and $\sS_0$ has an inseparable node in the sense of \cite{Posva_Gluing_for_surfaces_and_threefolds}.

\medskip
In our examples, the $1$-foliation $\sF$ is generated by (some variation of) the quadratic derivation $\partial=\sum_{i=1}^nx_i^2\partial_{x_i}$. The latter satisfies $\partial^{[p]}=0$ in every characteristic $p>0$, and thus defines an $\alpha_p$-action on $\bA^n$ such that $\bA^n/\alpha_p=\bA^n/\partial$ \cite[Proposition 2.3.7]{Posva_Singularities_quotients_by_1_foliations}. If $n=2=p$ then $\bA^2/\partial$ is the RDP $D^0_4$. The quadratic derivation has the nice property that a single blow-up simplifies its singularities into lc ones, making the necessary computations straightforward.

In principle, other choices of $1$-foliations $\sF$ can lead to similar examples, possibly with better dimensional bounds. The difficult part is the computation of the blow-up sequence $b\colon X\to\bA^n$ needed to simplify the singularities of $\sF$. The results of \cite{Posva_Resolution_of_1_foliations} suggests that such a sequence usually exists, but the number of necessary blow-ups might be quite large. To illustrate this, in \autoref{section:alternative_derivation} we consider the derivation $y^3\partial_x+x\partial_y +t\partial_z$ on $\bA^3_{x,y,z}\times\bA^1_t$: this choice was motivated by the fact that the quotient of $\bA^2$ by $y^3\partial_x+x\partial_y$ gives the RDP $E^0_6$ in characteristic $3$. The quotient gives a pathological family of affine threefolds in characteristic $3$ (\autoref{thm:loc_stable_families_II}). This family is more difficult to compactify, as explained in \autoref{rmk:compactification_is_difficult}: we shall not pursue this task here.

\begin{remark}
The underlying families $Y\to \bA^1$ (or $\sY\to \bA^1$) appearing in \autoref{thm_intro:local_stability}, \autoref{thm_intro:stability} and \autoref{thm:loc_stable_families_II} share the following properties: the central fibers $Y_0$ do not have liftable cohomology (in the sense of \cite{Kollar_Kovacs_Deformations_of_lc_and_F-pure_sing}), the pair $(Y,Y_0)$ is not $F$-pure along $Y_0$, and $Y$ is not $F$-injective at the non-$S_2$-points of $Y_0$. All the other fibers $Y_t$ $(t\neq 0)$ have liftable cohomology, and $(Y\setminus Y_0,Y_t)$ is $F$-pure (\footnote{
	Here is a sketch of proof. The fiber $Y_0$ does not have liftable cohomology, as otherwise it would be $S_2$ (even CM) by \cite[Theorem 8.5]{Kollar_Kovacs_Deformations_of_lc_and_F-pure_sing}. If $(Y,Y_0)$ was $F$-pure, then $\depth \sO_{Y,y}\geq \min\{3,\dim \sO_{Y,y}\}$ for every $y\in Y_0$ by \cite[Lemma 3.4]{Polstra_Simpson_Tucker_F-pure_inversion_of_adjunction}, and so $Y_0$ would be $S_2$. The fact that $Y$ is not $F$-injective is proved exactly as in \autoref{cor:not_F_inj}. On the other hand, if $t\neq 0$ then $Y_t$ has only $\mu_p$-quotient singularities and thus it is $F$-pure \cite[Theorem 1]{Posva_Singularities_quotients_by_1_foliations} : this implies that $Y_t$ has liftable cohomology \cite[Corollary 7.3]{Kollar_Kovacs_Deformations_of_lc_and_F-pure_sing} and that $(Y,Y_t)$ is $F$-pure in a neighbourhood of $Y_t$ \cite[Theorem A]{Polstra_Simpson_Tucker_F-pure_inversion_of_adjunction}.
}).
\end{remark}

\subsection{Consequences for KSBA moduli theory of stable pairs}\label{section:implications_moduli_theory}
The existence of the pathological stable families of \autoref{thm_intro:stability} has the following consequence, which extends the main result of \cite{Kollar_Families_of_3folds_in_pos_char}:

\begin{theorem_intro}[\autoref{thm:KSBA_mod_stacks}]\label{thm_intro:KSBA_mod_stacks}
Fix an integer $n\geq 3$. Then in characteristics $p\leq n$ the stable families of $n$-dimensional Cohen--Macaulay pairs (resp.\ of pairs with $F$-injective underlying varieties) usually do not compactify to a proper KSBA moduli stack.
\end{theorem_intro}

The formulation of that theorem requires an explanation. In positive characteristic, no precise definition of the KSBA moduli stack (\footnote{
		The distinction between moduli stacks and moduli functors is irrelevant for our discussion. While I use the language of stacks on a superficial level in this note, the reference for moduli theory of stable pairs in characteristic $0$ \cite{Kollar_Families_of_varieties_of_general_type} takes the viewpoint of moduli functors.
}) 
of stable pairs has been have been proposed so far in the literature: defining families of pairs over general bases is the delicate part. It is easier to define boundary-free families: to wit, the moduli stacks of stable surfaces are defined in \cite[\S 1.3]{Patakfalvi_Projectivity_moduli_space_of_surfaces_in_pos_char} and studied in \cite{Patakfalvi_Projectivity_moduli_space_of_surfaces_in_pos_char, Arvidsson_Bernasconi_Patakfalvi_Properness_moduli_surfaces, Posva_Gluing_for_surfaces_and_threefolds}. That definition can be extended to any dimension $d$, but then only few proprieties of the resulting stacks are known: for example, when $d\geq 3$ we do not know if they are algebraic. In any case, the boundary-free case is not our setting.

There is however no ambiguity about what the values of the KSBA moduli stacks should be over regular curves: stable families of pairs with demi-normal underlying fibers (see \autoref{def_app:stable_families}). If the moduli stacks would satisfy the valuative criterion of properness, then a stable family with demi-normal fibers over a punctured curve could be extended with a demi-normal fiber, possibly after replacing the base with a finite base-change. The examples of \autoref{thm_intro:stability} show that this is usually not possible when the characteristic is small compared to the dimension, even if the generic fiber has only klt, $F$-injective and Cohen--Macaulay singularities. 

This discussion and the content of \autoref{thm_intro:KSBA_mod_stacks} are made precise in \autoref{appendix}.

\medskip
Let us close this introduction with a few open questions related to KSBA moduli theory in positive characteristic:
\begin{question}
Let $k$ be an algebraically closed field of characteristic $p>0$. 
	\begin{enumerate}
		\item \emph{Are the poor properties of stable families of Cohen--Macaulay/$F$-injective pairs related to the ratio between the dimension and the characteristic?} More precisely, are there families of $n$-dimensional pairs such as those of \autoref{thm_intro:stability} if $p\gg n$?
		\item \emph{Do such pathologies also appear for stable families of $n$-dimensional varieties?} More precisely, are there stable families of $n$-dimensional varieties (without boundaries) over regular curves with pathologies such as those of \cite{Kollar_Families_of_3folds_in_pos_char} and \autoref{thm_intro:stability}? What if we assume that $p\gg n$, or that the generic fiber is Cohen--Macaulay or $F$-injective? 
		\item \emph{Are there pathological families of stable surface pairs in characteristics $\leq 5$?} 
	\end{enumerate}
\end{question}

\subsection{Acknowledgements}
The author is supported by the PostDoc Mobility grant P500PT-210980 from the Swiss National Science Foundation. I thank the Institute of Mathematics of the HHU Düsseldorf and Stefan Schröer's research group for their hospitality. I am also thankful to Zs. Patakfalvi and B. Totaro for several comments on early drafts to J. Baudin for discussions related to \cite{Baudin_Bernasconi_Kawakami_Frob_stable_GRV_fails}, and to the reviewer for several corrections.

\section{Preliminaries}

\subsection{Conventions and notations}
We work over an algebraically closed field $k$ of positive characteristic $p>0$.
	\begin{enumerate}
		\item A \emph{variety} (over $k$) is an integral quasi-projective $k$-scheme of finite type. A \emph{curve} (resp. \emph{surface}, \emph{threefold}) is $k$-variety of dimension one (resp. two, three).
		\item The conditions $S_i$ are the \emph{Serre's conditions}, see \cite[0341]{Stacks_Project}.
		\item Given a variety $X$, a divisor $E$ over $X$ is a prime divisor appearing on a (proper) birational model $\mu\colon Y\to X$. Its center $c_X(E)\subset X$ is the closure of $\mu(E)$.
		\item Let us recall some standard MMP terminology from \cite{Kollar_Mori_Birational_geometry_of_algebraic_varieties, Kollar_Singularities_of_the_minimal_model_program}. Let $X$ be a normal variety and $\Delta$ be a $\bQ$-divisor such that $K_X+\Delta$ is $\bQ$-Cartier. For a proper birational morphism $\mu\colon Y\to X$ with $Y$ normal we write
				$$K_Y =
				\mu^*(K_X+\Delta)+\sum_{E}
				a(E;X,\Delta)\cdot E.$$
		We define
				$$\discrep(X,\Delta)=\inf_E\{a(E;X,\Delta) \mid E\text{ exceptional over }X\}.$$
		We say that $(X,\Delta)$ is \emph{sub-lc} (resp.\ \emph{sub-klt}) if if for every $\mu$ and $E$ we have $a(E;X,\Delta)\geq -1$ (resp. $a(E;X,\Delta)>-1$). If $\Delta\geq 0$ and $(X,\Delta)$ is sub-lc (resp.\ sub-klt), we say that $(X,\Delta)$ is \emph{lc} (resp.\ \emph{klt}). We say that $(X,\Delta)$ is \emph{canonical} (resp.\ \emph{terminal}) if $\Delta$ is effective, and for every $\mu$ and $\mu$-exceptional $E$ we have $a(E;X,\Delta)\geq 0$ (resp. $a(E;X,\Delta)>0$).
		
		
		In case $(X,\Supp(\Delta))$ is log smooth, sufficient and necessary conditions on the coefficients of $\Delta$ for these conditions to hold are given in \cite[Corollary 2.11]{Kollar_Singularities_of_the_minimal_model_program}. 
In particular, a regular variety is terminal.

		\item We will several times use tacitly the following fact \cite[Lemma 2.5]{Kollar_Singularities_of_the_minimal_model_program}: if $K_X+\Delta$ is $\bQ$-Cartier, if $\mu\colon Y\to X$ is as birational with $Y$ normal, and if
				$$K_Y+\mu^{-1}_*\Delta+F=\mu^*(K_X+\Delta) \quad \text{where }F\text{ is }\mu\text{-exceptional},$$
	then $a(E;X;\Delta)=a(E;Y,\mu_*^{-1}\Delta+F)$ for every divisor $E$ over $X$.
	\end{enumerate}
	
\medskip
We also use the following terminology (see the \autoref{appendix} for more details):
\begin{definition}[{\cite[(2.3, 2.44)]{Kollar_Families_of_varieties_of_general_type}}]\label{def:local_stability}
Let $C$ be a normal $k$-variety of dimension $\leq 1$, let $f\colon (X,\Delta)\to C$ be a family of pairs \cite[Definition 2.2]{Kollar_Families_of_varieties_of_general_type} (in particular, the fibers $X_c$ are geometrically reduced), and assume that $X$ is normal. Then we say that $f$ is \emph{locally stable} if $(X,\Delta+X_\alpha)$ is lc for every $\alpha\in C(k)$. We say that $f$ is \emph{stable} if it is proper, locally stable and $K_X+\Delta$ is $f$-ample.
\end{definition}
	
We record the following well-known result.

\begin{lemma}\label{lemma:discrepancy_over_reg}
Let $X$ be a regular $k$-variety and $G$ be a divisor over $X$ with non-empty center $\mathfrak{c}=c_X(G)\subset X$. Then $a(G;X)\geq \codim_X(\mathfrak{c})-1$.
\end{lemma}
\begin{proof}
By a result of Zariski \cite[Lemma 2.45]{Kollar_Mori_Birational_geometry_of_algebraic_varieties}, there is a finite sequence of birational proper morphisms
		$$X_n\overset{f_n}{\longrightarrow} X_{n-1} 
		\overset{f_{n-1}}{\longrightarrow} \dots
		\overset{f_1}{\longrightarrow} X_0=X$$
such that each $f_i$ is the blow-up of $X_{i-1}$ along the center $\mathfrak{c}_{i-1}$ of $G$ on $X_{i-1}$, and $G$ appears on $X_n$ as an $f_n$-exceptional divisor dominating $\mathfrak{c}_{n-1}$. To prove the statement we may localize each $X_i$ in a neighbourhood of the generic point of $\mathfrak{c}_{i}$, and so we can assume that each $\mathfrak{c}_i$ is a regular subvariety and that each $X_i$ is regular. Let $E_i\subset X_i$ be the unique $f_i$-exceptional divisor: since $E_i=f_i^{-1}(\mathfrak{c}_{i-1})$ it holds that $\mathfrak{c}_i\subseteq E_i$ for each $i$, with equality if and only if $i=n$. By abuse of notation, the strict transform of $E_i$ on each $X_{i+1},\dots,X_n$ will also be denoted by $E_i$.

Let $f\colon X_n\to X$ be the composition of all blow-ups. We write
		$$f^*K_{X}=K_{X_n}-\sum_{i=1}^{n}a_iE_i.$$
We claim that $a_{i+1}>a_i$ for all $i$: since $a_1=\codim(\mathfrak{c})-1$ and $E_n=G$, the result will follow. These inequalities are easily proved by induction: we have
		$$f_i^*\left(K_{X_{i-1}}-\sum_{j=1}^{i-1}a_jE_j\right)
		= K_{X_i}-\sum_{j=1}^{i-1}a_jE_j
		-\left(\codim_{X_{i-1}}(\mathfrak{c}_{i-1})-1
		+\mult_{\mathfrak{c}_{i-1}}\sum_{j}a_jE_j\right)
		\cdot E_i.$$
Since $\mathfrak{c}_{i-1}\subsetneq E_{i-1}$ we have 
	$$a_i=\codim_{X_{i-1}}(\mathfrak{c}_{i-1})
		-1+\mult_{\mathfrak{c}_{i-1}}\sum_{j}a_jE_j
	\geq 1+a_{i-1}.$$
This proves the lemma.		
\end{proof}

\subsection{1-Foliations}
For basics on derivations and $1$-foliations, we refer to \cite{Posva_Singularities_quotients_by_1_foliations}. We recall below the most important facts for us. For derivations we need:

\begin{lemma}\label{lemma:p_power_of_derivation}
Let $\partial$ be a $p$-closed derivation on an integral scheme $\Spec(A)$. If there exists $a\in A$ such that $\partial(a)\neq 0$ and $\partial^{[p]}(a)=h\partial(a)$ for some $h\in A$, then $\partial^{[p]}=h\partial$.
\end{lemma}

For $1$-foliations, let us recall:

\begin{definition}
Let $\sF$ be a $1$-foliation on a normal variety $X$. Assume that the divisorial sheaf $\sO_X(K_\sF)=(\bigwedge^{\rank\sF}\sF)^{[-1]}$ is $\bQ$-Cartier. 
Then $\sF$ is \emph{log canonical (lc)} if for every proper birational morphism $\mu\colon Y\to X$ we can write
		$$K_{\mu^*\sF}=\mu^*K_\sF+\sum_{E: \ \mu\text{-exc.}} a(E;\sF) \cdot E$$
with 
		$$a(E;\sF)\geq
		\begin{cases}
		0 & \text{if }E\text{ is }\mu^*\sF\text{-invariant,}\\
		-1 & \text{otherwise.}
		\end{cases}$$
\end{definition}

\begin{definition}
Let $\sF$ be a $1$-foliation of rank $1$ on a normal variety $X$. Then $\sF$ has \emph{only multiplicative singularities} (\footnote{
	The definition \cite[Definition 2.4.7]{Posva_Singularities_quotients_by_1_foliations} is more general, but for $1$-foliations of rank $1$ the two definitions agree.
}) if for every $x\in \Sing(\sF)$, the restriction $\sF\otimes \widehat{\sO}_{X,x}$ is generated up-to-saturation by a continuous derivation $\partial\in \Der^{\text{cont}}_k \left(\widehat{\sO}_{X,x}\right)$ satisfying $\partial^{[p]}=u\partial$ where $u\in \widehat{\sO}_{X,x}^\times$.
\end{definition}

\begin{proposition}[{\cite[Corollary 1]{Posva_Singularities_quotients_by_1_foliations}}]\label{prop:mult_is_lc}
Let $X$ be a regular variety. If $\sF$ is a $1$-foliation of rank $1$, then $\sF$ has only multiplicative singularities if and only if $\sF$ is lc.
\end{proposition}

Pullbacks of $1$-foliations along smooth blow-ups can be computed as follows:
\begin{example}\label{example:pullback_smooth_blow-up}
Let $\pi\colon X\to \bA^n_{x_1,\dots,x_n}$ be the blow-up of the ideal $(x_1,\dots,x_r)$ for $r\leq n$. The $x_1$-chart $U_1$ of the blow-up $\pi^*\colon k[x_1,\dots,x_n]\to k[y_1,\dots,y_n]$ is given by 
				$$x_1\mapsto y_1, \quad
				x_i\mapsto y_1y_i \ (1<i\leq r), \quad
				x_j\mapsto y_j \ (j>r).$$
The pullbacks of derivations on $\bA^n$ by $\pi|_{U_1}$ can be computed using the following equalities:
				$$\pi^*\partial_{y_1}|_{U_1}=\partial_{y_1}-\sum_{i=2}^r\frac{y_i}{y_1}\partial_{y_i}, \quad
				\pi^*\partial_{x_i}|_{U_1}=\frac{1}{y_1}\partial_{y_i} \ (1<i\leq r), \quad
				\pi^*\partial_{x_j}|_{U_1}=\partial_{y_j} \ (j>r).$$
\end{example}

Let us recall some properties of quotients by $1$-foliations.

\begin{definition}[{\cite[Definition 4.2.1]{Posva_Singularities_quotients_by_1_foliations}}]
Let $X$ be a normal variety, $\sF$ a $1$-foliation on $X$ and $q\colon X\to X/\sF$ the quotient morphism. If $\Delta=\sum_i a_i\Delta_i$ is a $\bQ$-divisor on $X$, where $\Delta_i$ are prime divisors, then we let
		$$\Delta_{X/\sF}=\sum_i a_i \varepsilon_i q(\Delta_i)$$
where
		$$\varepsilon_i=
		\begin{cases}
		1 & \text{if }\Delta_i\text{ is }\sF\text{-invariant,}\\
		\frac{1}{p} & \text{otherwise.}
		\end{cases}$$
\end{definition}

\begin{proposition}[{\cite[Proposition 4.2.3]{Posva_Singularities_quotients_by_1_foliations}}]\label{prop:log_adjunction}
Let $X$ be a normal variety, $\Delta$ a $\bQ$-divisor on $X$, $\sF$ a $1$-foliation on $X$ and $q\colon X\to X/\sF$ the quotient morphism. Then we have an equality of $\bQ$-divisors
		$$q^*(K_{X/\sF}+\Delta_{X/\sF})=K_X+\Delta+(p-1)K_\sF.$$
\end{proposition}

The following theorem is our main tool to compute the discrepancies over quotient singularities.

\begin{theorem}\label{thm:sing_qt}
Let $(X,\Delta)$ be a normal sub-pair, $\sF$ be a $1$-foliation on $X$ and $q\colon (X,\Delta)\to (Y=X/\sF,\Delta_Y)$ be the associated quotient morphism. 
	\begin{enumerate}
		\item If $\sF$ is lc and $(X,\Delta)$ is sub-lc (respectively sub-klt), then $(Y,\Delta_Y)$ is sub-lc (respectively sub-klt).
	\end{enumerate}
Let $\mu\colon Y'\to Y$ be a birational proper morphism. Then:
	\begin{enumerate}\setcounter{enumi}{1}
		\item If $X'$ is the normalization of $Y'$ in $K(X)$, then we have a commutative diagram
			$$\begin{tikzcd}
			X'\arrow[r, "q'"]\arrow[d, "\mu'"] & Y'\arrow[d, "\mu"] \\
			X\arrow[r, "q"] & Y
			\end{tikzcd}$$
		where $\mu'$ is a proper birational morphism, and $q'$ is the quotient by $(\mu')^*\sF$.
		\item If $E\subset X'$ is a prime divisor with image the prime divisor $q'(E)=F\subset Y'$, then
				$$a(F; Y,\Delta_Y)=
				\begin{cases}
				a(E; X,\Delta)+(p-1)\cdot a(E; \sF) &
					\text{if }E\text{ is }(\mu')^*\sF\text{-invariant,} \\
				\frac{1}{p}\left[
				a(E; X,\Delta)+(p-1)\cdot a(E; \sF)\right]
				 & \text{otherwise}.
				\end{cases}$$
	\end{enumerate}
\end{theorem}
\begin{proof}
This is proved in \cite[Theorem 4.2.5]{Posva_Singularities_quotients_by_1_foliations}, under the additional assumption that $\Delta$ is effective. However, that assumption is not needed in the proof.
\end{proof}

Let us also recall the notion of family of $1$-foliations.
\begin{definition}[{\cite[\S 5.1]{Posva_Singularities_quotients_by_1_foliations}}]\label{def:family_foliations}
Let $X\to B$ be a flat morphism between normal varieties with geometrically normal fibers. A \emph{relative $1$-foliation} is a $1$-foliation on $X$ that is contained in the sub-module $T_{X/B}$ of $T_{X/k}$. If $X\to B$ is smooth, we say that a relative $1$-foliation $\sF\subset T_{X/B}$ is a \emph{family of $1$-foliations} if $T_{X/B}/\sF$ is flat over $B$ and for every $b\in B$ the sub-module $\sF\otimes \sO_{X_b}\hookrightarrow T_{X_b/k(b)}$ is a $1$-foliation.
\end{definition}

\begin{theorem}[{\cite[Proposition 5.2.4, Corollary 5.2.5]{Posva_Singularities_quotients_by_1_foliations}}]\label{thm:family_of_quotients}
Let $X\to B$ be a smooth morphism over a normal curve $B$, and $\sF\subset T_{X/B}$ be a family of $1$-foliations. Then:
	\begin{enumerate}
		\item for every $b\in B$ there is a natural morphism $\varphi_b\colon X_b/\sF|_{X_b}\to (X/\sF)_b$,
		\item $\varphi_b$ is an isomorphism if and only if $(X/\sF)_b$ is $S_2$, and
		\item if $\sF$ has only multiplicative singularities in a neighbourhood of $X_b$ then $\varphi_b$ is an isomorphism.
	\end{enumerate}
\end{theorem}

Finally, the following lemma will be useful for computations related to base-changes:
\begin{lemma}\label{lemma:pullback_relative_foliations}
Let $f\colon X\to B$ be a smooth morphism, with $B$ a normal variety. Let $g\colon B'\to B$ be a flat morphism from a normal variety, inducing the cartesian diagram
		$$\begin{tikzcd}
		X'\arrow[d, "f'"]\arrow[r, "g'"] & X\arrow[d, "f"] \\
		B'\arrow[r, "g"] & B.
		\end{tikzcd}$$
Let $\sF\subset T_{X/B}$ be a relative $1$-foliation, and $\sF'=(g')^*\sF\hookrightarrow (g')^*T_{X/B}\cong T_{X'/B'}$. Then $\sF'$ is a relative $1$-foliation, and the natural morphism $X'\to (X/\sF)\times_BB'$ is isomorphic over $B'$ to the quotient $X'\to X'/\sF'$.
\end{lemma}
\begin{proof}
This can be checked Zariski-locally on $X, B$ and $B'$, so we may assume that $X=\Spec(R)$, $B=\Spec(A)$, $B'=\Spec(A')$ and that $\sF=\tilde{F}$ is generated by $\partial_1,\dots,\partial_n\in \Der_A(R)$. Let $R'=R\otimes_A A'$. We have a natural map
		$$\mathfrak{g}'\colon \Der_R(A)\to \Der_{A'}(R'), \quad
		\mathfrak{g}'(\psi)\colon r\otimes a'\mapsto \psi(r)\otimes a',$$
and the $A'$-sub-module $F'$ of $\Der_{A'}(R')$ generated by $\mathfrak{g}'(F)$ sheafifies into $\sF'$. We check below that $F'$ is closed under Lie brackets and $p$-powers. As the $\partial_i'=\mathfrak{g}'(\partial_i)$ generate $F'$ we see that
		$$(R')^{F'}=
		\ker\left((\partial_1',\dots,\partial_n')\colon
		R'\to (R')^{\oplus n}\right).$$
Similarly we have
		$$R^F=\ker\left((\partial_1,\dots,\partial_n)\colon
		R\to R^{\oplus n}\right).$$
Since $A'$ is flat over $A$ we therefore obtain that $R^F\otimes_A A'=(R')^{F'}$, and the lemma follows.

It remains to check that $F'$ is closed under brackets and $p$-powers.
	Let $\xi,\psi\in \mathfrak{g}'(F)$ and $r,s\in R'$. As $\mathfrak{g}'$ commutes with Lie brackets we have
	$[r\xi, s\psi]=rs\cdot [\xi,\psi]\in F'$. 
	By Hochschild's formula \cite[(2.3.0.a)]{Posva_Singularities_quotients_by_1_foliations} we see that $(r\xi)^{[p]}\in F'$. Finally, a formula of Jacobson \cite[p. 209]{Jacobson_Abstract_derivations_and_Lie_algebras} shows that $(r\xi+s\psi)^{[p]}$ and $(r\xi)^{[p]}+(s\psi)^{[p]}$ differ by an $R'$-linear combination of multi-fold Lie brackets of $r\xi$ and $s\psi$. Thus $(r\xi+s\psi)^{[p]}\in F'$, and we are done. 
\end{proof}

\section{Non-$S_3$ isolated MMP singularities}
Consider the affine space $\bA^n$ over the field $k$ with coordinates $x_1,\dots,x_n$ where $n\geq 2$. We introduce the derivation 
		\begin{equation*}
		\partial = \sum_{i=1}^n x_i^2\partial_{x_i} \quad
		\text{on } \bA^n.
		\end{equation*}
First we observe that:
\begin{claim}\label{claim:p-closed}
$\partial^{[p]}=0$.
\end{claim}		
\begin{proof}\renewcommand{\qedsymbol}{$\lozenge$}
Since the summands $x_i^2\partial_{x_i}$ commute with each other, we have
		$$\partial^{[p]}=\sum_{i=1}^n \left( x_i^2\partial_{x_i}\right)^{[p]}.$$
Since $\partial_{x_i}$ is $p$-closed, by Hochschild's formula \cite[(2.3.0.a)]{Posva_Singularities_quotients_by_1_foliations} the scaling $x_i^2\partial_{x_i}$ is also $p$-closed. It is easily seen that $(x_i^2\partial_{x_i})^{[p]}(x_i)=0$, and by \autoref{lemma:p_power_of_derivation} it follows that $(x_i^2\partial_{x_i})^{[p]}=0$. Therefore $\partial^{[p]}=0$.
\end{proof}		
		
In particular $\sF=\sO_{\bA^n}\cdot \partial$ is a $1$-foliation on $\bA^n$.

\begin{theorem}\label{thm:non_S3_qt}
With the notations as above, let $Y=\bA^n/\sF$ (with $n\geq 2$). Then:
	\begin{enumerate}
		\item $Y$ is a normal $\bQ$-factorial variety, with a unique singular closed point $\bold{0}$;
		\item $\sO_{Y,\bold{0}}$ is not $S_3$ if $n\geq 3$;
		\item It holds that:
			\begin{itemize}
				\item If $n\leq p-2$, $Y$ is not lc.
				\item If $n\geq p-1$, $Y$ is lc;
				\item If $n\geq p$, $Y$ is canonical;
				\item If $n\geq p+1$, $Y$ is terminal.
			\end{itemize}
	\end{enumerate}
\end{theorem}
\begin{proof}
By construction $Y$ is normal and $\bQ$-factorial: in particular $K_Y$ is $\bQ$-Cartier. Since $\Sing(\sF)$ is the origin $0\in\bA^n$, we see that $Y$ is regular away from the image $\bold{0}$ of the origin \cite[Lemma 2.5.10]{Posva_Singularities_quotients_by_1_foliations}. As the ideal generated by the image of $\partial$ in the local ring $\sO_{\bA^n,0}$ is $\fm_{\bA^n,0}$-primary, by \autoref{claim:p-closed} and \cite[Remark 2.5.4]{Posva_Singularities_quotients_by_1_foliations} we obtain that $\sO_{Y,\bold{0}}=\sO_{\bA^n,0}^\partial$ is not $S_3$.

It remains to analyse the singularity of $(\bold{0}\in Y)$ through the glasses of the MMP. To begin with, we blow-up the origin of $\bA^n$ to simplify the singularities of $\sF$:
		$$b\colon X=\Bl_0\bA^n\longrightarrow \bA^n.$$
Consider the $x_1$-patch $U_1$ of this blow-up:
		$$b^*\colon k[x_1,\dots,x_n]\to k[u_1,\dots,u_n],\quad
		x_1\mapsto u_1, \ x_i\mapsto u_1u_i \ (i>1).$$
Using the blow-up computations of \autoref{example:pullback_smooth_blow-up}, we see that on this chart we have
		\begin{eqnarray*}
		b^*\partial|_{U_1} & = &
		u_1 \cdot \left(u_1\partial_{u_1}-\sum_{i=2}^nu_i\partial_{u_i}\right)
		+ \sum_{i=2}^n u_1^2u_i^2 \cdot \frac{1}{u_1}\partial_{u_i} \\
		& = & u_1 \cdot \underbrace{\left[
		u_1\partial_{u_1} + \sum_{i=2}^n (-u_i+u_i^2)\partial_{u_i}
		\right]}_{\psi}
		\end{eqnarray*}
Since $b^*\partial$ is $p$-closed, its scaling $\psi$ is also $p$-closed. As $\psi(u_1)=u_1$ we obtain by \autoref{lemma:p_power_of_derivation} that $\psi^{[p]}=\psi$. Therefore by \autoref{prop:mult_is_lc} we obtain that
		\begin{equation*}
		\sO_{U_1}\cdot \psi =(b^*\sF)|_{U_1} \quad \text{is an lc 1-foliation.}
		\end{equation*}
The situation is similar on the other charts, and therefore we find that $b^*\sF$ is an lc $1$-foliation. As $\psi(u_1)\in (u_1)$ we also see that the exceptional divisor $E\subset X$ of $b$ is $b^*\sF$-invariant. Finally, as $b^*\partial|_{U_1}=u_1\cdot \psi$ we see that $a(E;\sF)=-1$.

Now let $Z=X/b^*\sF$ and consider the induced commutative diagram
		\begin{equation}\label{eqn:pseudo_resolution}
		\begin{tikzcd}
		X\arrow[d, "b"] \arrow[r, "q'"] & Z\arrow[d, "\fb"]\\
		\bA^n\arrow[r, "q"] & Y
		\end{tikzcd}
		\end{equation}
where the horizontal arrows are the quotient morphisms and where $\fb\colon Z\to Y$ is birational. Let $F\subset Z$ be the unique prime $\fb$-exceptional divisor. By \autoref{thm:sing_qt} we have
		$$a(F;Y)=a(E;\bA^n)+(p-1)\cdot a(E; \sF)
		= (n-1)+(p-1)(-1)
		=n-p.$$
If $n\leq p-2$ then $a(F;Y)\leq -2$ and $Y$ is not lc. From now on assume that $n\geq p-1$. We have:
		\begin{equation*}
		K_Z+(p-n)F=\fb^*K_Y.
		\end{equation*}
This crepant relation shows that $Y$ is lc if and only if $\discrep(Z,(p-n)F)\geq -1$. Since $b^*\sF$ is lc, \autoref{thm:sing_qt} shows that $\discrep(Z,(p-n)F)\geq -1$ as soon as $\discrep(X,(p-n)E)\geq -1$. As $(X,E)$ is log smooth, this holds when $p-n\leq 1$ by \cite[Corollary 2.11]{Kollar_Singularities_of_the_minimal_model_program}. This settles the lc case of the theorem.

It remains to study when $Y$ is canonical (resp. terminal): for this we assume that $n\geq p$. Consider an exceptional divisor $F'$ over $Y$, appearing on a birational model $Z'\to Y$. Replacing $Z'$ by some blow-up, we  may assume that $Z'\to Y$ factors through a birational morphism $\ff\colon Z'\to Z$. Let $f\colon X'\to X$ be the normalization of $\ff$ in $K(X)$, and $E'\subset X'$ be the prime divisor dominating $F'\subset Z'$. If $f(E')\not\subset E$ then $c_{F'}(Y)\neq \bold{0}$, and therefore $a(F';Y)\geq 1$ as $Y$ is regular away from $\bold{0}$. Thus we may assume that $f(E')\subset E$. By \autoref{thm:sing_qt} we have
		\begin{eqnarray*}
		a(F';Y) &=& a(F';Z,(p-n)F) \\
		& = & \delta_{E'} \cdot \left[
			a(E'; X, (p-n)E) + (p-1)\cdot a(E';b^*\sF)\right]
		\end{eqnarray*}
where $\delta_{E'}\in \{1, p^{-1}\}$. We must find conditions that guarantee that $a(F';Y)\geq 0$ (resp. $>0$), hence the factor $\delta_{E'}$ will not play any role. As $(X,E)$ is log smooth and $p-n\leq 0$, we have $a(E';X,(p-n)E)\geq 1$. Since regular $1$-foliations are canonical \cite[Lemma 3.0.3]{Posva_Singularities_quotients_by_1_foliations}, it follows that $a(F';Y)>0$ unless the center of $E'$ is contained in the singular locus of $b^*\sF$: so from now on we assume that $c_X(E')$ is contained in $\Sing(b^*\sF)$. It is easily seen, using the computations above, that the singular locus of $b^*\sF$ is a finite collection of isolated points. Therefore, by \autoref{lemma:discrepancy_over_reg} we see that $a(E';X)\geq n-1$. Moreover, as $b^*\sF$ is lc we have $a(E';b^*\sF)\geq -1$. Putting everything together, we get
		\begin{eqnarray*}
			\frac{1}{\delta_{E'}}a(F';Y)
			& = &
			a(E';X)+(n-p)\cdot \mult_{E'}(f^*E) +(p-1)\cdot a(E';b^*\sF) \\
			&\geq & (n-1) + (n-p) + (p-1)\cdot (-1) \\
			& = & 2(n-p).
		\end{eqnarray*}
	So if $n\geq p$ (resp. if $n>p$) it holds that $a(F';Y)\geq 0$ (resp. $a(F';Y)>0$). This completes the proof.
\end{proof}

\begin{remarks}
\begin{enumerate}
	\item The case $n=3\geq p$ gives a canonical non-Cohen--Macaulay isolated threefold singularity. It was shown in \cite{Arvidsson_Bernasconi_Lacini_KVV_for_log_dP_surfaces_in_pos_char} that klt threefold singularities (with perfect residue fields) are Cohen--Macaulay in characteristic $p>5$. This is optimal, as there also exists non-Cohen--Macaulay terminal threefold singularities in characteristic $5$ \cite{Totaro_Terminal_non_CM_3folds}.
	\item In characteristic $0$, strictly lc singularities may not be $S_3$, but this can only happen if they are log canonical centers \cite[Corollary 7.21]{Kollar_Singularities_of_the_minimal_model_program}. If $n=p-1$ in \autoref{thm:non_S3_qt} then $\bold{0}$ is an lc center of $Y$, as the proof shows.
\end{enumerate}
\end{remarks}

We end this section with some further remarks on the singularities $(\bold{0}\in Y)$ constructed in \autoref{thm:non_S3_qt}. Unless we specify otherwise, there will be no constraint on the characteristic $p$ nor on $n=\dim Y$.

\begin{proposition}\label{prop:not_1_Gor}
If $p>2$, then the singularity $(\bold{0}\in Y)$ is never $1$-Gorenstein (i.e.\ $K_Y$ is never Cartier).
\end{proposition}
This is false if $p=2$: for example, in dimension two we obtain the RDP $D^0_4$, which is Gorenstein.
\begin{proof}
Let us use the notation of the proof of \autoref{thm:non_S3_qt}. We perform a weighted blow-up of the closed point defined by $(u_1,\dots,u_n)$ in $U_1\subset X$ with weights $(1,p-1,\dots,p-1)$ (see \cite[Example 2.4.3]{Posva_Resolution_of_1_foliations}). A (schematic) affine chart of the weighted blow-up is given by
		$$k[u_1,\dots,u_n]\to k[v_1,\dots,v_n],\quad
		u_1\mapsto v_1, \ u_i\mapsto v_1^{p-1}v_i \ (i\geq 2).$$
The (unique) exceptional divisor $E'$ is given by $(v_1=0)$. We have
		$$du_1\wedge\dots\wedge du_n=
		v_1^{(n-1)(p-1)}dv_1\wedge\dots\wedge dv_n$$
and thus $a(E';X)=(n-1)(p-1)$. Since
		$$u_1\partial_{u_1}=v_1\partial_{v_1}-\sum_{i=2}^n(p-1)v_i\partial_{v_i}, \quad
		u_i\partial_{u_i}=v_i\partial_{v_i} \ (i\geq 2),$$
we find
		$$\psi=v_1\cdot \left[
		\partial_{v_1}+ \sum_{i=2}^n v_1^{p-2}v_i^2\partial_{v_i}\right].$$
Therefore $a(E';b^*\sF)=-1$ and $E'$ is \emph{not} invariant with respect to the pullback of $b^*\sF$. So if $F'$ is the divisor over $Y$ induced by $E'$, we find:
		\begin{eqnarray*}
		a(F';Y) &=& a(F';Z, (p-n)F) \\
		&=& \frac{1}{p}\cdot\left[ a(E'; X, (p-n)E)+(p-1)\cdot a(E';b^*\sF)\right] \\
		&=& \frac{1}{p}\cdot \left[
		(n-1)(p-1)-(p-n)+(p-1)(-1)\right] \\
		&=& n-3 + \frac{2}{p}.
		\end{eqnarray*}
Thus when $p>2$ we see that $a(F';Y)\notin \bZ$. This shows that $Y$ is not $1$-Gorenstein.
\end{proof}

For the definition of $W\sO$-rationality and Cohen--Macaulayness up to (Frobenius) nilpotence used in the following proposition, and the notions that appear in its proof, we refer to \cite{Baudin_Bernasconi_Kawakami_Frob_stable_GRV_fails}.

\begin{proposition}\label{prop:CM_up_to_nilpotence}
The singularities $(\bold{0}\in Y)$ are $W\sO$-rational and Cohen--Macaulay up to nilpotence.
\end{proposition}
\begin{proof}
The property of $W\sO$-rationality descends finite universal homeomorphisms by \cite[Proposition 4.4.9]{Chatzistamatiou_Rulling_Hodge-Witt_coh_and_Witt-rat_sing}. It evidently holds for $\bA^n$, and so it holds by $Y$.

Let us show that $Y$ is CM up to nilpotence. By \cite[Lemma 2.3]{Baudin_Bernasconi_Kawakami_Frob_stable_GRV_fails} the natural morphism $\sO_Y\to Rq_*\sO_{\bA^n}$ is an isomorphism in $D^b(\Crys^F_Y)$. Let $\omega_Y^\bullet$ be a dualizing complex on $Y$: by \cite[Theorem 4.3.5]{Baudin_Duality_between_Cartier_crystals_and_perverse_sheaves}, by applying $\RHom(-,\omega_Y^\bullet)$ to the previous isomorphism we obtain
		$$\RHom(Rq_*\sO_{\bA^n},\omega_Y^\bullet)\cong
		\RHom(\sO_Y,\omega_Y^\bullet) \quad
		\text{in }D^b(\Crys^C_Y).$$
Now we have
		\begin{eqnarray*}
		\RHom(Rq_*\sO_{\bA^n},\omega_Y^\bullet) &\cong &
		Rq_*\RHom(\sO_{\bA^n},\omega_{\bA^n}) \\
		&\cong & Rq_*\omega_{\bA^n} \\
		&\cong & q_*\omega_{\bA^n}
		\end{eqnarray*}
in $D^b(\Crys^C_Y)$, where the first isomorphism holds by \cite[Corollary 5.1.7]{Baudin_Duality_between_Cartier_crystals_and_perverse_sheaves} and the third one holds as $q$ is finite. This shows that $\RHom(\sO_Y,\omega_Y^\bullet)$ is supported in a single degree as an element of the category $\Crys^C_Y$. By \cite[Lemma 3.2]{Baudin_Bernasconi_Kawakami_Frob_stable_GRV_fails} it follows that $\sO_Y$ is CM up to nilpotence as an element of $\Crys_Y^F$. 
\end{proof}

\begin{corollary}\label{cor:not_F_inj}
The singularities $(\bold{0}\in Y)$ are never $F$-injective.
\end{corollary}
\begin{proof}
This follows from \cite[Lemma 2.5.4]{Posva_Singularities_quotients_by_1_foliations}, but here is another proof that utilizes the notions introduced above. Since $Y$ is CM up to nilpotence, the action of Frobenius on the local cohomology groups $H^i_{\fm}(\sO_{Y,\bold{0}})$, where $\fm$ is the maximal ideal and $0\leq i< \dim \sO_{Y,\bold{0}}$, is nilpotent (see \cite[Lemma 3.2]{Baudin_Bernasconi_Kawakami_Frob_stable_GRV_fails}). Since $\sO_{Y,\bold{0}}$ is $S_2$ but not $S_3$, we have $H^2_{\fm}(\sO_{Y,\bold{0}})\neq 0$. Thus the Frobenius action on $H^2_{\fm}(\sO_{Y,\bold{0}})$ is not injective. This shows that $\sO_{Y,\bold{0}}$ is not $F$-injective.
\end{proof}

\begin{proposition}\label{prop:rational_up_to_Frob}
There exist proper birational morphisms $\pi\colon V\to Y$ where $V$ is a regular variety. Moreover, for any such:
	\begin{enumerate}
		\item the Frobenius action on $R^i\pi_*\sO_V$ is nilpotent for $i>0$ \emph{(\footnote{This is equivalent to say that $Y$ has $\bF_p$-rational singularities, see \cite[Remark 3.5]{Baudin_Bernasconi_Kawakami_Frob_stable_GRV_fails}. Combining this statement, \autoref{prop:CM_up_to_nilpotence} and \cite[Proposition 3.8]{Baudin_Bernasconi_Kawakami_Frob_stable_GRV_fails}, we obtain that $Y$ satisfies the \emph{Frobenius stable Grauert--Riemenschneider vanishing} \cite[Definition 3.6]{Baudin_Bernasconi_Kawakami_Frob_stable_GRV_fails}.
		})}, and
		\item the Frobenius trace action on $R^i\pi_*\omega_V$ is nilpotent for $i>0$.
	\end{enumerate}
\end{proposition}
\begin{proof}
We begin by constructing a specific resolution of singularities and checking the nilpotence statements along it. Consider the diagram \autoref{eqn:pseudo_resolution} in the proof of \autoref{thm:non_S3_qt}. The $1$-foliation $b^*\sF$ is locally generated by multiplicative derivations with isolated singularities. Thus the quotient $Z=X/b^*\sF$ has only isolated $\mu_p$-singularities \cite[Proposition 4.1.1]{Posva_Singularities_quotients_by_1_foliations}. These toroidal singularities can be resolved using \cite[Theorem 11* p.94]{KKMSD_Toroidal_embeddings}: let $\mu\colon V'\to Z$ be the induced resolution of singularities. The composition $\pi'=\fb\circ \mu\colon V'\to Y$ is also a resolution of singularities.

We claim that the sheaves $R^i\mu_*\sO_{V'}$ and $R^i\mu_*\omega_{V'}$ vanish for $i>0$. This can be checked formally on $Z$, and thus we reduce to the toric situation. Then vanishing holds for $i>0$ by \cite[Theorem 14 p. 52]{KKMSD_Toroidal_embeddings}. From the Leray spectral sequence it follows that
		$$R^i\pi'_*\sO_{V'}=R^i\fb_*\sO_{Z}\quad \text{and}\quad
		R^i\pi'_*\omega_{V'}=R^i\fb_*\omega_{Z} \quad \text{for } i\geq 0.$$
By naturality of the Frobenius, these equalities reduce the nilpotence statements along $\pi'$ to the nilpotence statements along the morphism $\fb\colon Z\to Y$. By construction the absolute Frobenius morphism $F_Z\colon \sO_Z\to F_*\sO_Z$ factors through $q'_*\sO_X$. Therefore $R^i\fb_*(F_Z)$ factors through $R^i\fb_*(q'_*\sO_X)=q_*R^ib_*\sO_X$. Since $b$ is the blow-up of a point of $\bA^n$ we have $R^ib_*\sO_X=0$ for $i>0$, which shows that $R^i\fb_*(F_Y)$ is the zero map for $i>0$. By applying (derived) Hom into $\omega_Z$ to the map $F_Z$, one finds dually that the Frobenius trace $\Tr_Z\colon F_*\omega_Z\to \omega_Z$ factors through $q'_*\omega_X$. As before we have $R^i\fb_*(q_*'\omega_X)=q_*R^ib_*\omega_X=0$ for $i>0$, which shows that $R^i\fb_*(\Tr_Z)$ is the zero map for $i>0$. This proves the nilpotence statements along $\fb$ and $\pi'$.

Finally, let $\pi\colon V\to Y$ be another proper birational morphism with $V$ regular. By \cite[Theorem 1]{Chatzistamatiou_Rulling_Higher_direct_images_in_pos_char} we have $R^i\pi_*\sO_V\cong R^i\pi'_*\sO_{V'}$ and $R^i\pi_*\omega_V\cong R^i\pi'_*\omega_{V'}$ for every $i\geq 0$. So the nilpotence statements along $\pi'$ imply the corresponding nilpotence statements along $\pi$, which concludes the proof.
\end{proof}

\section{Locally stable families with non-$S_2$ special fibers}
\subsection{First example: $\mathbf{\dim\geq \max\{p,3\}.}$}
Consider the affine space $\bA^{n+1}_{x_1,\dots,x_n,y}\times\bA^1_t$ and the derivation
		\begin{equation}\label{eqn:family_quadratic_derivation}
		\partial_m=\sum_{i=1}^{n}x_i^2\partial_{x_i} + t^m\mu(y)\partial_{y} \quad \text{on }\bA^{n+1}\times\bA^1,
		\end{equation}
where $m\geq 0$ and $\mu(y)\in k[y]$ satisfies $(\mu \partial_{y})^{[p]}=0$ and becomes a unit in a neighbourhood of the origin, i.e.\ $\mu(0)\neq 0$. 

We let $\bA=\bA^n_{\bold{x}}\times D(\mu(y))\subset \bA^{n+1}$ be the affine open subset where $\mu(y)$ is invertible. We let $\sF_m\subset T_{(\bA\times \bA^1)/k}$ be the sub-module generated by $\partial_m$. 

(We can of course take $m=1$, $\mu=1$ and $\bA=\bA^{n+1}$; but this extra generality will be useful for \autoref{thm:stable_families} below.)

\begin{claim}\label{claim:F_m_family_of_foliations}
$\sF_m$ is a family of $1$-foliations over $\bA^1_t$ (\autoref{def:family_foliations}) if and only if $n\geq 2$.
\end{claim}
\begin{proof}\renewcommand{\qedsymbol}{$\lozenge$}
Clearly $\partial_m\in T_{(\bA\times\bA^1)/\bA^1}$, and we see that $\partial_m^{[p]}=0$ as in \autoref{claim:p-closed} (and this survives along any specialization of $t$). Thus $\sF_m$ is a relative $1$-foliation over $\bA^1_t$, for any $n\geq 1$. It is easily seen that the cokernel of $\sF_m\hookrightarrow T_{\bA\times \bA^1/\bA^1}$ is torsion-free, hence flat, over $\bA^1$.

The restriction of $\sF_m$ to the generic fiber of the projection $\bA\times\bA^1_t\to \bA^1_t$ is the restriction to $\bA\times\Spec(k(t))$ of the module generated by
		$$\sum_{i=1}^nx_i^2\partial_{x_i}+t^m\mu(y)\partial_y\in 
		\Der_{k(t)}k(t)[x_1,\dots,x_n,y],$$
which is clearly a $1$-foliation for every $n\geq 1$. A similar computation shows that $\sF_m|_{(t=\alpha)}$ is a $1$-foliation for every $0\neq \alpha\in k$ and $n\geq 1$. The only condition is given by the restriction to the central fiber $(t=0)$: indeed, $\sF_m|_{(t=0)}$ is generated by
		$$\sum_{i=1}^nx_i^2\partial_{x_i}\in 
		\Der_{k}k[x_1,\dots,x_n,y],$$
so $\sF_m|_{(t=0)}$ is a $1$-foliation if and only $n\geq 2$. (For $n=1$, observe that $x_1^2\partial_{x_1}$ does not generate a saturated sub-module of $\Der_{k}k[x_1,\dots,x_n,y]$.) The statement follows.
\end{proof}

The following lemma explains how to simplify the singularities of $\sF_m$.

\begin{lemma}\label{lemma:blowup_to_reduce_degree}
We have $\Sing(\sF_m)=V(x_1,\dots,x_n,t)$ for $m\geq 1$, and $\sF_0$ is regular. If $m\geq 1$ and $b\colon X=\Bl_{\Sing(\sF_m)}\to \bA\times\bA^1$ is the blow-up of that singular locus with exceptional divisor $E$, then:
	\begin{enumerate}
		\item $a(E;\sF_m)=-1$;
		\item $E$ is $b^*\sF_m$-invariant;
		\item $b^*\sF_m$ is lc on the $x_i$-patches, and isomorphic over $\bA^1_t$ to $\sF_{m-1}$ on the $t$-patch;
		\item if $m\geq 2$, the locus of non-lc singularities $S$ of $b^*\sF_m$ is closed in $X$ and contained on the $t$-patch;
		\item $\mult_SE=1$, $\codim_XS=n+1$ and $S$ is disjoint from the strict transform of $(t=0)\subset \bA\times\bA^1_t$.
	\end{enumerate}
\end{lemma}
\begin{proof}
The statement about the singular locus is clear. On the $x_1$-patch $U_1$ the blow-up $b$ is given by
			\begin{eqnarray*}
				b^*\colon && k[x_1,\dots,x_n,y,t][\mu(y)^{-1}]\longrightarrow k[u_1,\dots,u_n,y,s][\mu(y)^{-1}], \\
				&& x_1\mapsto u_1, \quad
				x_i\mapsto u_1u_i \ (1<i\leq n), \quad
				y\mapsto y,\quad
				t\mapsto u_1s.
				\end{eqnarray*}
			Therefore we have
					\begin{eqnarray*}
					b^*\partial_m|_{U_1} &=&
					u_1\cdot \left(u_1\partial_{u_1} 
						-\sum_{i=2}^{n} u_i\partial_{u_i}
						- s\partial_s\right)
					+ \sum_{i=2}^{n} u_1^2u_i^2 \cdot \frac{1}{u_1}\partial_{u_i}
					+ u_1^ms^m \mu(y) \cdot \partial_{y} \\
					&=& u_1 \cdot \underbrace{\left[
					u_1\partial_{u_1}
					+\sum_{i=2}^{n} (-u_i+u_i^2)\partial_{u_i}
					+u_1^{m-1}s^m\mu(y)\partial_{y}
					-s\partial_s
					\right]}_{\psi}
					\end{eqnarray*}
Since $b^*\partial_m|_{U_1}$ is $p$-closed, so is $\psi$. As $\psi(u_1)=u_1$, by \autoref{lemma:p_power_of_derivation} we obtain that $\psi^{[p]}=\psi$. Thus $b^*\sF_m|_{U_1}=\sO_{U_1}\cdot \psi$ is lc by \autoref{prop:mult_is_lc}. Moreover, the above computations show that $E$ is $b^*\sF_m$-invariant and that $a(E;\sF_m)=-1$. The situation on the other $x_i$-patches $U_i$ is similar.
			
On the $t$-patch $V$ the blow-up $b$ is given by 								\begin{eqnarray*}
				b^*\colon && k[x_1,\dots,x_n,y,t][\mu(y)^{-1}]\longrightarrow k[v_1,\dots,v_n,y,s][\mu(y)^{-1}], \\
				&& x_i\mapsto sv_i \ (1\leq i\leq n), \quad
				y\mapsto y, \quad
				t\mapsto s.
				\end{eqnarray*}
			Therefore we have:
				\begin{eqnarray*}
				b^*\partial_m|_V &=& 
				\sum_{i=1}^n v_i^2s \partial_{v_i} + s^m\mu(y)\partial_{y} \\
				&=& s\cdot \underbrace{\left[
				\sum_{i=1}^{n}v_i^2\partial_{v_i}+s^{m-1}\mu(y)\partial_{y}
				\right]}_{\xi}
				\end{eqnarray*}
Under the description $V\cong D_{\bA^{n+1}_{v_1,\dots,v_n,y}}(\mu(y))\times\bA^1_s$, we see that $\xi$ corresponds to $\partial_{m-1}$ and thus $b^*\sF_m|_V\cong \sF_{m-1}$. To finish the proof, it remains to observe that $\Sing(\xi)=V(v_1,\dots,v_n,s)\subset V$ is in fact a closed subset of $X$, and that the strict transform of $(t=0)$ is disjoint from the patch $V$.
\end{proof}

\begin{theorem}\label{thm:loc_stable_families_I}
With the notations as above, let $Y=(\bA\times\bA^1)/\sF_m$ with $n\geq 2$ and $m\geq 1$. Then:
	\begin{enumerate}
		\item The projection $\bA\times \bA^1_t\to\bA^1_t$ factors through a flat morphism $Y\to \bA^1_t$ of relative dimension $n+1$;
		\item $Y_t$ is smooth for $t\neq 0$, while $Y_0$ is reduced and non-$S_2$;
		\item If $n+1\geq \max\{p,3\}$, the family $Y\to \bA^1$ is locally stable.
	\end{enumerate}
\end{theorem}
The hypothesis $n\geq 2$ ensures that the $1$-foliation $\sF_m$ it is a family of $1$-foliations over $\bA^1_t$, see \autoref{claim:F_m_family_of_foliations}. The case $(p=2,n=1)$ is considered in \autoref{example:extreme_case_I} below. 
\begin{proof}
Since $\partial_m(t)=0$ we have $k[t]\subset k[x_1,\dots,x_n,y,t]^{\partial_m}$: this proves the first point. On the principal open set $D(t)\subset \bA\times \bA^1$ the $1$-foliation $\sF_m|_{D(t)}$ is regular. Thus for $0\neq \alpha\in k$, by \autoref{thm:family_of_quotients} we have
		$$Y_\alpha=\bA/\sF_{m,\alpha} \quad\text{for}\quad
		\sF_{m,\alpha} = \sO_{\bA}\cdot \left(
		\sum_{i=1}^{n}x_i^2\partial_{x_i} + \alpha^m\mu(\alpha)\partial_{y}
		\right).$$
As $\sF_{m,\alpha}$ is regular for $\alpha\neq 0$ we obtain that $Y_\alpha$ is smooth \cite[Lemma 2.5.10]{Posva_Singularities_quotients_by_1_foliations}. 

Let us prove that $Y_0$ is not $S_2$. For this we reduce to $m=1$. Indeed, consider the cartesian diagram
		$$\begin{tikzcd}
		\bA\times\bA^1_t \arrow[r, "v"] \arrow[d] & \bA\times \bA^1_{t'}\arrow[d] \\
		\bA^1_t\arrow[r, "u"] & \bA^1_{t'}
		\end{tikzcd}$$
where the vertical arrows are the projections and $u\colon \bA^1_t\to\bA^1_{t'}$ is given by $u^*(t')=t^m$. By \autoref{lemma:pullback_relative_foliations} we have $v^*\sF_1=\sF_m$ (where $\sF_1$ is defined as in \autoref{eqn:family_quadratic_derivation} using $t'$ instead of $t$) and an $\bA^1_t$-isomorphism
		$$(\bA\times \bA^1_{t'})/\sF_1 \times_{\bA^1_{t'}}\bA^1_t
		\cong Y.$$
Hence it suffices to show that $Y_0$ is non-$S_2$ in the case $m=1$. 
As $\sF_1$ is a family of $1$-foliations over $\bA^1_t$, by \autoref{thm:family_of_quotients} it is sufficient to prove that
		$$\sO_{\bA\times \bA^1,\bold{0}}^{\partial_1}
		\otimes_{k[t]} k[t]/(t)
		\subsetneqq
		\left(\sO_{\bA,\bold{0}}/(t)\right)^{\bar{\partial}}
		\quad \text{with}\quad 
		\bar{\partial}=\sum_{i=1}^{n}x_i^2\partial_{x_i},$$
where $\bold{0}=V(x_1,\dots,x_n,y,t)$.
Clearly $\bar{\partial}(y)=0$, so it suffices to show that there is no $f\in \sO_{\bA\times \bA^1,\bold{0}}$
such that $\partial_1(y+tf)=0$. If there was such an $f$, we would have
		$$0=\partial_1(y+tf)=t\mu +\sum_{i=1}^n tx_i^2\partial_{x_i}(f)+t^2\mu \partial_y(f).$$
Taking in account that $t$ is not a zero-divisor in $\sO_{\bA\times \bA^1,\bold{0}}$, we would obtain
		$$-\mu = \sum_{i=1}^nx_i^2\partial_{x_i}(f)+t\mu \partial_y(f).$$
However the right-hand side belongs to the prime ideal $(x_1,\dots,x_{n},t)$, so this equality cannot hold in the local ring $\sO_{\bA\times \bA^1,\bold{0}}$ by our choice of $\mu$. We have reached a contradiction, and hence $Y_0$ is not $S_2$ at the image of $\bold{0}$.

On the other hand, I claim that $Y_0$ is (geometrically) reduced. Since $Y$ is $S_2$, the Cartier divisor $Y_0$ is $S_1$ and it is sufficient to show that $Y_0$ is generically reduced. Observe that $Y_0$ is irreducible. The $1$-foliation $\sF_m$ is regular on the open set $U=D(x_1)$, which intersects $Y_0$ non-trivially. By \autoref{thm:family_of_quotients} it follows that the open subset $(U/{\sF_m}|_U)_0$ of $Y_0$ is the quotient of $U_0$ by (the restriction of) $\sF_{m,0}$. In particular it is normal, so $Y_0$ is generically reduced.

It remains to show that $(Y,Y_0)$ is lc (in the general case $m\geq 1$). For this, as in the proof of \autoref{thm:non_S3_qt}, we blow-up $\bA\times\bA^1$ until the pullback of $\sF_m$ has only lc singularities. By \autoref{lemma:blowup_to_reduce_degree}, this produces a sequence of $m$ smooth blow-ups
		$$X_m\longrightarrow X_{m-1}\longrightarrow \cdots
		\longrightarrow X_1\longrightarrow W=\bA\times\bA^1$$
such that, if $b_m\colon X_m\to W$ is the composition of the blow-ups, the $1$-foliation $b_m^*\sF_m$ is lc. Let $E_j\subset X_m$ be the strict transform of the exceptional divisor of $X_j\to X_{j-1}$ (for $1\leq j\leq m$), and denote $(t=0)\subset W$ by $W_0$. We have $a(E_1;W,W_0)=n-1$, and using \autoref{lemma:blowup_to_reduce_degree} we compute
		\begin{eqnarray*}
		a(E_2;W,W_0)&=&a(E_2;X_1, (X_1\to W)^{-1}_*W_0 + (1-n)E_1)\\
		&=& 2n-1.
		\end{eqnarray*}
Continuing by induction, we find
		$$a(E_j;W,W_0)=jn-1, \quad j=1,\dots,m.$$
Now let $Z=X_m/b_m^*\sF_m$ and consider the induced commutative diagram
		$$\begin{tikzcd}
		X_m\arrow[d, "b_m"] \arrow[r] & Z\arrow[d, "\fb"]\\
		\bA\times\bA^1\arrow[r] & Y
		\end{tikzcd}$$
where the horizontal arrows are the quotient morphisms and where $\fb\colon Z\to Y$ is birational. Let $F_j\subset Z$ be the prime divisor which is the image of $E_j$. Using \autoref{lemma:blowup_to_reduce_degree} we get
		$$a(E_j;\sF_m)=-j, \quad j=1,\dots,m.$$
Now we use \autoref{thm:sing_qt}: since $W_0$ is $\sF_m$-invariant we find
		\begin{eqnarray*}
		a(F_j;Y,Y_0) &=& a(E_j; W,W_0)+(p-1)\cdot a(E_j; \sF_m) \\
		&=& j(n-p+1)-1.
		\end{eqnarray*}
Therefore we have the crepant equation
		$$K_Z+\fb^{-1}_*Y_0-\sum_{j=1}^m(j(n-p+1)-1)F_j
		=\fb^*(K_Y+Y_0).$$
To show that $(Y,Y_0)$ is lc it suffices to show that 
	$$\discrep\left(Z,\fb^{-1}_*Y_0-\sum_{j=1}^m(j(n-p+1)-1)F_j\right)\geq -1.$$ 
Since $b_m^*\sF_m$ is lc, by \autoref{thm:sing_qt} it suffices to show that 
	$$\discrep\left(X_m,(b_m)^{-1}_*W_0-\sum_{j=1}^m(j(n-p+1)-1)E_j\right)\geq -1.$$ 
As $(X_m,(b_m^{-1})_*W_0+\sum_{j=1}^mE_j)$ is log smooth, this holds by \cite[Corollary 2.11]{Kollar_Singularities_of_the_minimal_model_program} as soon as $1-j(n-p+1)\leq 1$ for every $j\geq 1$. This is the case when $n+1\geq p$. The proof is complete.
\end{proof}

\begin{remark}
The normalization $Y_0^\nu$ of the $n+1$-dimensional central fiber $Y_0$ in \autoref{thm:loc_stable_families_I} is isomorphic to the product of $\bA^1$ with the $n$-dimensional quotient singularity of \autoref{thm:non_S3_qt}. In particular, it is $\bQ$-factorial, $S_3$ but non-$S_4$, non-$F$-injective, with a singular locus of dimension one. Since $n+1\geq \max\{p,3\}$ we have $n\geq p-1$ and so $Y_0^\nu$ is lc. If $n\geq p$ (resp.\ $n>p$), it is even canonical (resp.\ terminal).
\end{remark}

\begin{corollary}\label{cor:canonical_non_S3}
Let $Y=\bA^{n+2}_{x_1,\dots,x_{n},y,t}/\langle \sum_{i=1}^nx_i^2\partial_{x_i}+t\partial_y\rangle$ with $n\geq 2$. If $\dim Y=n+2\geq p+1$, then $Y$ has non-$S_3$ canonical singularities along the image of $V(x_1,\dots,x_n,t)$.
\end{corollary}
\begin{proof}
The computations in the proof of \autoref{thm:loc_stable_families_I} show that the central fiber $Y_0$ of $Y\to \bA^1_t$ is not $S_2$ along the image of $V(x_1,\dots,x_n,t)$. Thus $Y$ is non-$S_3$ along that curve. To show that $Y$ is canonical, as the open set $Y\setminus Y_0$ is regular and thus terminal, it suffices to show that $a(E;Y)\geq 0$ for every divisor $E$ over $Y$ with $c_Y(E)\subseteq Y_0$. As $Y_0$ is an effective Cartier divisor and $(Y,Y_0)$ is lc we have 
	$$a(E;Y)=a(E;Y,Y_0)+\mult_{E}Y_0\geq -1+1=0$$
and we are done.
\end{proof}

\begin{example}\label{example:extreme_case_I}
Suppose that $k$ has characteristic $p=2$, and consider the derivation $\partial=x^2\partial_x+t\partial_y$ on $\bA^3_{x,y,t}$. Then $\partial^{[2]}=0$ and that $X=\bA^3/\partial$ is a flat $\bA^1_t$-scheme given by the spectrum of $k[x^2,y^2,t,tx+x^2y]$: indeed, the latter ring is clearly $\partial$-invariant, is normal and has degree $4$ over $k[x^2,y^2,t^2]$.

It is easily seen that $X$ is isomorphic to an hypersurface in $\bA^4$, and thus $X$ is Cohen--Macaulay. Its central fiber $X_0=\Spec k[x^2,y^2,x^2y]\cong V(W^2-X^2Y)\subset \bA^3_{X,Y,W}$ is demi-normal but not normal: it has a nodal singularity along the curve $(x=0)$. In fact, this is an inseparable node \cite[\S 3.1]{Posva_Gluing_for_surfaces_and_threefolds}, and thus $X_0$ is semi-normal but not weakly normal \cite[Lemma 2.3.7]{Posva_Gluing_surfaces_mixed_char}. Simple computations in the spirit of the proof of \autoref{thm:loc_stable_families_I} show that the pair $(X,X_0)$ is lc but not plt, as the image of the line $(x=0=t)$ in $X$ is the singular locus of $X_0$ and an lc center of $(X,X_0)$. Compare also with the general semi-normality results of \cite{Arvidsson_Posva_Normality_min_lc_centers_of_3folds}.
\end{example}


\autoref{thm:loc_stable_families_I} gives a series of local examples. We can compactify them as explained in the next theorem to obtain global ones.

\begin{theorem}\label{thm:stable_families}
For each $p>0$, there exists a projective family $(\sY,\sB=\frac{1}{p}\sH)\to \bA^1$ of relative dimension $N\geq \max\{p,3\}$ such that:
	\begin{enumerate}
		\item\label{item:a} $(\sY,\frac{1}{p}\sH)\to \bA^1$ is a stable family,
		\item\label{item:b} $\sY_t$ has only $\mu_p$-quotient singularities for $t\neq 0$, and
		\item\label{item:c} $\sY_0$ is not $S_2$.
	\end{enumerate}
Moreover:
	\begin{itemize}
		\item There exists a finite purely inseparable morphism $\bP^{N-1}\times E\times \bA^1\to \sY$ of degree $p$ over $\bA^1$, where $E$ is a supersingular elliptic curve, and $\sH$ is the image of a $\bQ$-divisor $H_0\times \bA^1\subset \bP^{N-1}\times E\times \bA^1$;
		\item If $p=2$ then we can modify the construction so that $\sY_t$ is regular for every $t\neq 0$;
		\item If $C$ is a normal curve and $C\to\bA^1$ is a finite flat morphism, then $\sY_C=\sY\times_{\bA^1}C$ is normal and $(\sY_C, \frac{1}{p}\sH_C)\to C$ is stable.
	\end{itemize}
\end{theorem}

The following lemma will be useful for the proof.

\begin{lemma}\label{lemma:global_foliation}
Consider $E$ a supersingular elliptic curve, $\omega\in H^0(E,T_E)$ a global generator, $\bA^{N-1}_{x_1,\dots,x_{N-1}}\subset \bP^{N-1}$ a standard chart (with $N\geq 3$), $A$ an integral Noetherian $k$-algebra and $a\in A$. Then the derivation
		$$\partial_{(0)}=\sum_{i=1}^{N-1}x_i^2\partial_{x_i} +a\omega \quad
		\text{on} \quad \bA^{N-1}\times E\times \Spec(A)$$
induces a $1$-foliation $\sG$ of rank $1$ on $\bP^{N-1}\times E\times \Spec(A)$ such that:
	\begin{enumerate}
		\item $K_{\sG}=\sO_{\bP^{N-1}}(1)\boxtimes \sO_{E\times\Spec(A)}$, 
		\item $\sG$ has only multiplicative singularities away from $\{[1:0:\dots :0]\}\times E \times V(a)$,
		\item If $A=k[t]$ and $a=t^m$ with $m\geq 1$, then $\sG\subset T_{(\bP^{N-1}\times E\times \bA^1)/\bA^1}$ is a family of $1$-foliations.
	\end{enumerate}	
\end{lemma}
\begin{proof}
Since $E$ is supersingular we have $\omega^{[p]}=0$ \cite[12.4.1.3]{Katz_Mazur_Arithmetic_elliptic_curves}. An argument similar to \autoref{claim:p-closed} shows that $\partial_{(0)}$ is $p$-closed. Therefore it induces a $1$-foliation $\sG$ of rank $1$ on $\bP^{N-1}\times E\times \Spec(A)$. We compute the expression of $\partial_{(0)}$ on the other charts. Let $X_0,\dots,X_{N-1}$ be homogeneous coordinates of $\bP^{N-1}$: we may assume that $x_i=X_i/X_0$ for $1\leq i\leq N-1$. If $u_i=X_i/X_{N-1}$ for $0\leq i<N-1$, then we have
		\begin{eqnarray*}
		x_i=\frac{u_i}{u_0} \ (1\leq i<N-1), \quad x_{N-1}=\frac{1}{u_0},\\
		u_i=\frac{x_i}{x_{N-1}} \ (1\leq i< N-1), \quad u_0=\frac{1}{x_{N-1}}.
		\end{eqnarray*}
Therefore
		$$\partial_{x_i}=u_0\partial_{u_i} \ (1\leq i<N-1),\quad
		\partial_{x_{N-1}}=-u_0^2\partial_{u_0}-\sum_{i=1}^{N-2}u_0u_i\partial_{u_i}$$
and so
		\begin{equation}\label{eqn:change_of_coordinates}
		\partial_{(0)}=\frac{-1}{u_0}\cdot \underbrace{\left[
			u_0\partial_{u_0}+\sum_{i=1}^{N-2}(u_i-u_i^2)\partial_{u_i} - au_0\omega
		\right]}_{\partial_{(N-1)}}.
		\end{equation}
This shows that $\sG|_{D_+(X_{N-1})\times E\times\Spec(A)}$ is generated by $\partial_{(N-1)}$ and that $K_{\sG}^{-1}=\sG\cong \sO_{\bP^{N-1}}(-1)\boxtimes \sO_{E\times\Spec(A)}$. This proves the first assertion.

On $D_+(X_{N-1})\times E\times \bA^1$, notice that $\partial_{(N-1)}$ is $p$-closed and that $\partial_{(N-1)}(u_0)=u_0$. Therefore $\partial_{(N-1)}^{[p]}=\partial_{(N-1)}$ by \autoref{lemma:p_power_of_derivation}, which shows that $\sG$ has only multiplicative singularities on $D_+(X_{N-1})\times E\times\bA^1$. The computations on the patches $D_+(X_i)\times E\times\bA^1$ for $1\leq i<N-1$ are similar. On $D_+(X_0)\times E\times\bA^1$, we see that the singular locus of $\sG$ is $\{[1:0:\dots : 0]\}\times E\times V(a)$. This proves the second assertion. In case $A=k[t]$ and $a=t^m$, we proceed as in \autoref{claim:F_m_family_of_foliations} over each charts to see that $\sG$ is a family of $1$-foliations if and only if $N\geq 3$, which holds by assumption. 
\end{proof}

\begin{remark}\label{rmk:can_we_get_smth_family?}
The multiplicative singularities of $\sG$ cannot be resolved by smooth blow-ups (unless $p=2$). Indeed, the singular locus of $\partial_{(N-1)}$ is the union of the closed subsets
		$$Z_{\varepsilon_\bullet}=[\varepsilon_0=0 : \varepsilon_1: \dots \varepsilon_{N-2} : 1]\times E\times \Spec(A)$$
where $\varepsilon_1,\dots,\varepsilon_{N-2}\in\{0,1\}$. It follows from \cite[Proposition 4.1.1]{Posva_Singularities_quotients_by_1_foliations} that along $Z_{\varepsilon_\bullet}$, and formal-locally on its first factor, the $1$-foliation $\sG$ is generated by
		$$\sum_{i: \ \varepsilon_i=0} \bar{u}_i\partial_{\bar{u}_i}
		-\sum_{j: \ \varepsilon_j=1} \bar{u}_j\partial_{\bar{u}_j}
		-a\bar{u}_0\omega$$
where $\bar{u}_\bullet$ are formal parameters of $\sO_{\bP^{N-1},[\varepsilon_\bullet:1]}$ with $\bar{u}_0=u_0$. To resolve the singularity of $\sG$ along $Z_{\varepsilon_\bullet}$, we need to take the weighted blow-up given by the Rees algebra
		$$\sum_{i: \ \varepsilon_i=0} (\bar{u}_i,1)+
		\sum_{j: \ \varepsilon_j=1} (\bar{u}_j, p-1).$$
As soon as some weight $p-1\neq 1$ appears, this blow-up produces a non-schematic tame regular Deligne--Mumford stack whose coarse moduli space has quotient singularities: see \cite{Posva_Resolution_of_1_foliations} for details. So the singularities of $\sG$ along $Z_{\varepsilon_\bullet}$ cannot be resolved by usual blow-ups, unless $\varepsilon_\bullet=(0,\dots,0)$ or $p=2$. The situation on the other patches $D_+(X_i)\times E\times \Spec(A)$, for $1\leq i<N-1$, is similar.
\end{remark}

\begin{proof}[Proof of \autoref{thm:stable_families}]
For clarity, we divide the proof in several steps. If $p=2$ we can use a sligthly different construction, which we describe in Step 5.

\medskip
\textsc{Step 1: Construction of $\sY\to\bA^1$.} 
Let $E$ be a supersingular elliptic curve and $\omega\in H^0(E,T_E)$ be a global generator. If $X_0,\dots,X_{N-1}$ are the homogeneous coordinates on $\bP^{N-1}$, consider on $D_+(X_0)\times E\times \bA^1_t$ the derivation
		$$\partial_{(0)}=\sum_{i=1}^{N-1}x_i^2\partial_{x_i}+t\omega, \quad 
		x_i=X_i/X_0.$$
By \autoref{lemma:global_foliation} it induces a $1$-foliation $\sG$ on $W=\bP^{N-1}\times E\times\bA^1$. We let $\sY=(\bP^{N-1}\times E\times\bA^1)/\sG$, with quotient morphism $q\colon W\to\sY$. By construction the projection $W\to \bA^1$ factors through a flat projective morphism $\sY\to \bA^1$.

\medskip
\textsc{Step 2: Construction of the divisor $\sH$.} 
The divisor $\sH$ is chosen as follows. Take a $\bQ$-divisor $H$ on $W$ such that:
	\begin{itemize}
		\item $H=H_0\times\bA^1$ for some $\bQ$-divisor $H_0$ on $\bP^{N-1}\times E$;
		\item $K_W+(p-1)K_\sG+H$ is ample over $\bA^1$,
		\item $(W,W_\alpha+H)$ is lc for every $\alpha\in \bA^1\setminus\{0\}$, 
		\item the support of $H$ lies in general position with respect to the curve $\{[1:0:\dots : 0]\}\times E\times \{0\}$, and
		\item $H$ is not $\sG$-invariant.
	\end{itemize}
The last three properties are satisfied for a general choice of divisor of the form $H_0\times \bA^1$, thus by choosing $H_0$ general in a sufficiently ample linear system on $\bP^{N-1}$, we see there exists an $H$ with all these properties. Notice that $K_W+(p-1)K_\sG\cong \sO_{\bP^{N-1}}(-N+p-1)\boxtimes \sO_{E\times\bA^1}$ by \autoref{lemma:global_foliation}, so because of the second property and on the assumption $N\geq \max\{p,3\}$ we cannot take $H=0$. 

We let $\sH\subset \sY$ be the prime divisor with support $q(H)$. 

\medskip
\textsc{Step 3: Behaviour under base-change.}
Let us immediately observe what happens when we base-change $\sY\to\bA^1_t$ along a flat finite morphism $g\colon C\to\bA^1_t$. By \autoref{lemma:pullback_relative_foliations} we have a natural isomorphism
		$$\sY\times_{\bA^1}C\cong 
		\left( \bP^{N-1}\times E\times C\right)/f^*\sG$$
where $f\colon \bP^{N-1}\times E\times C\to \bP^{N-1}\times E\times \bA^1$ is the base-change of $g$. This shows that $\sY_C$ is normal. Observe that $H_C=f^*H$ still satisfies the five properties listed in Step 2. The pullback of $K_\sY+\frac{1}{p}\sH$ along $\sY_C\to \sY$ is equal to $K_{\sY_C}+\frac{1}{p}\sH_C$: therefore the latter is ample over $C$ as soon as $K_\sY+\frac{1}{p}\sH$ is ample over $\bA^1$. By \autoref{prop:log_adjunction} and the non-$\sG$-invariance of $H$ we have
		\begin{equation}\label{eqn:pullback_formula}
		q^*\left(K_\sY+\frac{1}{p}\sH\right)=K_W+(p-1)K_\sG+H
		\end{equation}
which, by choice of $H$, is ample over $\bA^1$. 

To complete the proof, it remains to check properties \autoref{item:b} and \autoref{item:c}, and that $(\sY_C, \frac{1}{p}\sH_C)\to C$ is locally stable. Properties \autoref{item:b} and \autoref{item:c} can be checked after an arbitrary surjective base-change. For local stability, by \cite[2.15.5]{Kollar_Families_of_varieties_of_general_type} and \cite{Hu_Zong_Base_change_local-stability_positive_char} it is enough to check it in case $g=F^r_{\bA^1/k}$ is an iterate of the $k$-linear Frobenius morphism of $\bA^1$ (\footnote{
		Instead of appealing to \cite{Kollar_Families_of_varieties_of_general_type, Hu_Zong_Base_change_local-stability_positive_char}, one could proceed as follows. The local stability is Zariski-local over $C'$, so we reduce to base-changes of the form $\Spec(\sO) \to \bA^1_t$ where $\sO$ is a DVR. If $s$ is a uniformizer of $\sO$, write $t=\nu(s)s^m$ where $\nu(s)\in \sO^\times$. Then one can prove the analogue of \autoref{thm:loc_stable_families_I} for derivations of the form $\sum_{i=1}^{n}x_i^2\partial_{x_i} + s^m\nu(s)\mu(y)\partial_{y}$: the argument is the same, but the bookkeeping is even more cumbersome. The rest of the proof of \autoref{thm:stable_families} will then be similar.
}). 

\medskip
\textsc{Step 4: Local stability and singularities of fibers.} 
To summarize, we have reduced to prove the following: let $q'\colon W=\bP^{N-1}\times E\times \bA^1_t\to \sY'$ be the quotient by the $1$-foliation $\sG'$ induced by the derivation
		$$\partial_{(0)}'=\sum_{i=1}^{N-1}x_i^2\partial_{x_i}+t^{p^r}\omega, \quad 
		\text{on}\quad
		D_+(X_0)\times E\times \bA^1_t.$$
It is endowed with a flat proper morphism $\sY'\to \bA^1_t$. Let $H'$ be a $\bQ$-divisor on $W$ satisfying the properties listed in Step 2: we let $\sH'=q(H')$. We must show that $(\sY',\frac{1}{p}\sH')\to\bA^1$ is locally stable, that $\sY'_t$ has only $\mu_p$-quotient singularities for $t\neq 0$, and that $\sY_0'$ is not $S_2$. As we will see, everything is a consequence of \autoref{thm:loc_stable_families_I} and \autoref{lemma:global_foliation}.

\medskip
\textsc{Step 4.1: Singularity of the central fiber.} 
First, let us prove that $\sY_0'$ is not $S_2$ along the image through $q'$ of the proper curve $\{[1:0:\dots :0]\}\times E\times \{0\}$. Indeed, let $z=\{[1:0:\dots :0]\}\times \{e\}\times \{0\}$ where $e$ is an arbitrary closed point of $E$. As $\sY_0'$ is a Cartier divisor in $\sY$, it is equivalent to show that $\sO_{\sY',q'(z)}$ is not $S_3$. I claim that this is \'{e}tale-local over $z$. For if $h\colon \Spec(\sO_{W,z}^h)\to \Spec(\sO_{W,z})$ is the henselization, by \cite[Lemma 3.0.7]{Posva_Resolution_of_1_foliations} we have a cartesian diagram
		$$\begin{tikzcd}
		\Spec(\sO_{W,z}^h)\arrow[r, "h"] \arrow[d] & \Spec(\sO_{W,z})\arrow[d, "q'"] \\
		\Spec(\sO_{W,z}^h)/h^*\sG'\arrow[r, "g"] & \Spec(\sO_{\sY',q'(z)})
		\end{tikzcd}$$
where $g$ is an \'{e}tale morphism. Therefore it suffices to prove that $\Spec(\sO_{W,z}^h)/h^*\sG'$ is not $S_3$. If $y\in \sO_{E,e}^h$ is an \'{e}tale coordinate at $e\in E$, then an \'{e}tale-local generator for $\sG'$ at $z$ is given by
		$$\partial^h=\sum_{i=1}^{p}x_i^2\partial_{x_i} + t^{p^r} \mu(y) \partial_y$$
where $\mu(y)\in \sO_{E,e}^h$ is a unit such that $\omega\otimes \sO_{E,e}^h=\mu(y)\partial_y$. Then we apply \autoref{thm:loc_stable_families_I} (or rather its formal-local version at the origin, which is proved similarly) to see that the $\partial^h$-invariant sub-ring of $\sO_{W,z}^h$ is not $S_3$. 

Since $\sY'$ is $S_2$, the irreducible Cartier divisor $\sY_0'$ is nonetheless $S_1$. As it is generically reduced by \autoref{thm:loc_stable_families_I}, we obtain that $\sY_0'$ is (geometrically) reduced.

\medskip
\textsc{Step 4.2: Local stability.} 
Now we prove that $(\sY',\frac{1}{p}\sH')\to\bA^1$ is locally stable. First, let $0\neq \alpha\in k$. As $\sG'|_{D(t)}$ is lc and $(W,W_\alpha+H')$ is lc, by \autoref{thm:sing_qt} we have that $(\sY',\sY'_\alpha+\frac{1}{p}\sH')$ is lc along $\sY_\alpha$. Moreover, as $\sG'|_{D(t)}$ has only multiplicative singularities by \autoref{lemma:global_foliation}, by \autoref{thm:family_of_quotients} we obtain that
		$$\sY'_\alpha= (\bP^{N-1}\times E)/\sG'_\alpha$$
where $\sG'_\alpha$ is the $1$-foliation induced by $\sum_{i=1}^{N-1}x_i^2\partial_{x_i}+\alpha^{p^r}\omega$. By \autoref{lemma:global_foliation} we see that $\sG'_\alpha$ has only multiplicative singularities. Therefore $\sY_\alpha$ has only $\mu_p$-quotient singularities by \cite[Proposition 4.1.1]{Posva_Singularities_quotients_by_1_foliations}. 

It remains to prove that $(\sY',\sY'_0+\frac{1}{p}\sH')$ is lc. To this end, we blow-up $W$ to simplify the singularities of $\sG'$. By \autoref{lemma:global_foliation} the non-lc singularities of $\sG'$ are contained on the chart $D_+(X_0)\times E\times \bA^1$, where (\'{e}tale-locally on the factor $E$) the $1$-foliation $\sG'$ is generated by a derivation of the form \autoref{eqn:family_quadratic_derivation} with $m=p^r$. Therefore, the proofs of \autoref{lemma:blowup_to_reduce_degree} and \autoref{thm:loc_stable_families_I} show that if we blow-up repeatedly the locus of non-lc singularities of $\sG'$, we obtain a birational proper morphism
		$$b\colon X=X_{p^r}\longrightarrow W$$
with $p^r$ exceptional divisors $E_1,\dots,E_{p^r}\subset X$, such that $b^*\sG'$ is lc, $(X,b_*^{-1}W_0+\sum_j E_j)$ is log smooth, each $E_j$ is $b^*\sG'$-invariant, $a(E_j;\sG')=-1$ and $a(E_j;W,W_0)=j(N-1)-1$. By our choice of $H'$, we also have that $(X,b_*^{-1}(W_0+H')+\sum_jE_j)$ is log smooth and
		$$a(E_j;W,W_0+H')=a(E_j; W,W_0) \quad \forall j=1,\dots,p^r.$$
The arguments are now the same as the ones at the end of the proof of \autoref{thm:loc_stable_families_I}. Consider the commutative diagram
		$$\begin{tikzcd}
		X\arrow[d, "b"]\arrow[r] & Z=X/b^*\sG'\arrow[d, "\fb"] \\
		W\arrow[r, "q'"] & \sY'.
		\end{tikzcd}$$
If $F_1,\dots,F_{p^r}\subset Z$ denote the $\fb$-exceptional divisors, then by \autoref{thm:sing_qt} we have
		\begin{eqnarray*}
		a\left(F_j; \sY', \sY_0'+\frac{1}{p}\sH'\right) &=& 
		a(E_j; W, W_0+H')+(p-1)\cdot a(E_j;\sG') \\
		&=& a(E_j; W, W_0) +(p-1)(-j) \\
		&=& j(N-p)-1
		\end{eqnarray*}		 
Hence we have the crepant equation
		$$K_Z+\fb^{-1}_*\left(\sY_0'+\frac{1}{p}\sH'\right)+\sum_{j=1}^{p^r}(1-j(N-p))F_j
		=\fb^*\left(K_{\sY'}+\sY'_0+\frac{1}{p}\sH'\right).$$
Notice that $1-j(N-p)\leq 1$ for every $j\geq 1$ is equivalent to $N\geq p$, which holds as $N=\max\{p,3\}$. Therefore we have that 
	$$\discrep\left(X,b^{-1}_*(W_0+H)+\sum_{j=1}^{p^r}(1-j(N-p))E_j\right)\geq -1.$$
Hence by \autoref{thm:sing_qt} we obtain that 
	$$\discrep\left(Z,\fb^{-1}_*(\sY_0+\frac{1}{p}\sH)+\sum_{j=1}^{p^r}(1-j(N-p))F_j\right)\geq -1,$$
and therefore $(\sY',\sY'_0+\frac{1}{p}\sH')$ is lc. 

\medskip
\textsc{Step 5: the case $p=2$.} 
It follows from \autoref{rmk:can_we_get_smth_family?} that there exists a smooth blow-up $Z\to \bP^{N-1}$ whose (reducible) center does not contain $[1:0:\dots :0]$, and such that if $b\colon Z\times E\times \bA^1\to W$ is the induced morphism, the birational pullback $b^*\sG$ is regular away from the preimage of $\{[1:0: \dots : 0]\}\times E\times \{0\}$. We take 
		$$\sZ=(Z\times E\times \bA^1)/b^*\sG\longrightarrow \bA^1$$
as the example. By \cite[Proposition 5.2.4, Lemma 2.5.10]{Posva_Singularities_quotients_by_1_foliations} we have that $\sZ_\alpha$ is regular for $\alpha\neq 0$. The rest of the properties are proved exactly as above. This completes the proof.
\end{proof}

Let us make some observations about the preceding proof.

\begin{remark}\label{rmk:boundary_divisor}
The coefficients of $H$ are $\leq 1$, so the coefficients of $\frac{1}{p}\sH$ are $\leq 1/p$. The extra factor $1/p$ appears because $H$ is chosen to be non-$\sG$-invariant. I do not know whether there exists a $\sG$-invariant divisor $H$ that satisfies the other properties. On the other hand, the number of irreducible components of $H$ can be any (in particular, $\sH$ can be irreducible).
\end{remark}

\begin{remark}\label{rmk:volume}
If $e\in E$ an arbitrary point then the line bundle $\sL=\sO_{\bP^{N-1}}(N+p)\boxtimes \sO_E(e)$ is ample on $\bP^{N-1}\times E$, and $\sL^{\otimes n}$ is very ample for every $n\geq 3$. Moreover, for any integer $m\geq 1$, if $H_0$ is a general element of the $\bQ$-linear system $\frac{1}{m}|\sL^{\otimes n}|$ then $H_0\times\bA^1$ satisfies the conditions listed in Step 2.
\end{remark}

\subsection{Second example: $\mathbf{\dim=3, p=3}.$}\label{section:alternative_derivation}
Consider the affine space $\bA^3_{x,y,z}\times\bA^1_t$ and the derivation
		$$\partial_m=y^3\partial_x+x\partial_y+t^m\partial_z, \quad m\geq 1.$$
Using an argument as in \autoref{claim:p-closed} together with \cite[Corollary 6]{Mitsui_Sato_Criterion_for_p_closedness_in_dim_2}, one checks that $\partial_m^{[p]} =0$ as soon as $p>2$. In that case $\partial_m$ generates a $1$-foliation $\sF_m$ on $\bA^3\times\bA^1$. As in \autoref{claim:F_m_family_of_foliations}, one sees that $\sF_m$ is a family of $1$-foliations over $\bA^1_t$.

Let us show how to simplify the singularities of $\sF_m$. As in \autoref{lemma:blowup_to_reduce_degree}, we are able to reduce $\sF_m$ to $\sF_{m-1}$, but we need two blow-ups to achieve this.

\begin{claim}\label{claim:multiplicative_derivation}
For $p>2$, we have $((1-2x^2)\partial_x)^{[p]}=2^{(5p-3)/2} (1-2x^2)\partial_x$.
\end{claim}
\begin{proof}
After the change of variable $u=1+\sqrt{2}x$, we reduce to consider the $p$-th power of the derivation $\partial=\sqrt{2}u(2-u)\partial_u$ on $k[u]$. By a formula of Jacobson \cite[p. 209]{Jacobson_Abstract_derivations_and_Lie_algebras}, we have
		$$\partial^{[p]}=2^{p/2}\left((-u^2\partial_x)^{[p]}+(2u\partial_u)^{[p]}+S\right)$$
where $S$ is a $k$-linear combination of $(p-1)$-fold commutators of $-u^2\partial_x$ and $2u\partial_u$. Now we make several observations. First, by Hoschild's formula $\partial$ is $p$-closed and so $\partial^{[p]}$ is a $k[u]$-scaling of $\partial$. As $(-u^2\partial_u)^{[p]}=0$ and $(2u\partial_u)^{[p]}=2^{p}u\partial_u$, it follows that $S\neq 0$. As
		$$[u^2\partial_u,\ 2u\partial_u]=-2u^2\partial_u$$
we see that $S\neq 0$ can only be of the form $-\alpha u^2\partial_u$ for some $\alpha\in k^\times$. Hence $\partial^{[p]}=2^{p/2}(2^pu\partial_u-\alpha u^2\partial_u)$ is a $k[u]$-scaling of $\partial=2^{1/2}(2u\partial_u-u^2\partial_u)$, which forces $\alpha=2^{(5p-3)/2}$. 
\end{proof}

\begin{lemma}\label{lemma:blow-up_reduce_degree_II}
Suppose that $p>2$. We have $\Sing(\sF_m)=V(x,y,t)$ for $m\geq 1$, and $\sF_0$ is regular. If $m\geq 1$ and
		$$b_1\colon Y=\Bl_{(x,y,t)}(\bA^3\times \bA^1)\longrightarrow \bA^3\times\bA^1$$
with exceptional divisor $E$, then:
	\begin{enumerate}
		\item $a(E;\sF_m)=0$,
		\item $E$ is $b_1^*\sF_m$-invariant,
		\item the locus $S$ of (non-lc) singularities of $b_1^*\sF_m$ is closed, $\mult_{S}E=1$, $\codim_YS=2$ and $S$ is not contained in $(b_1)^{-1}_*(t=0)$.
	\end{enumerate}
Consider $b_2\colon \Bl_SY\longrightarrow Y$ with exceptional divisor $F$ and $b=b_1\circ b_2$. Then:
	\begin{enumerate}
		\item $a(F;\sF_m)=-1$,
		\item $F$ is $b^*\sF_m$-invariant,
		\item if $m\geq 2$, the non-lc locus of $b^*\sF_m$ is closed, contained in a blow-up chart $V$ of $b_2$, not contained in the support of $b_*^{-1}(t=0)+(b_2)^{-1}_*E$, and
		\item we have an $\bA^1_t$-isomorphism $V\cong \bA^3\times \bA^1_t$ under which $b^*\sF_m|_V$ is isomorphic to $\sF_{m-1}$, and the $b_2$-exceptional divisor corresponds to $(t=0)$.
	\end{enumerate}
\end{lemma}
\begin{proof}
Let us first consider the blow-up $b_1$. We have three blow-up charts, which we denote by $U_{A}$, $U_B$ and $U_C$:
	\begin{itemize}
		\item \textsc{Chart $U_A$}, given by $x\mapsto u, \ y\mapsto uv, \ z\mapsto z, \ t\mapsto us$. Then 
					$$b_1^*\partial_m|_{U_A}= u^3v^3\partial_u+(1-u^2v^4)\partial_v-u^2v^3s\partial_s+u^ms^m\partial_z$$
			which shows that $b_1^*\sF_m|_{U_A}$ is a regular $1$-foliation. Notice that $b_1^*\partial_m(u)\in (u)$, so $E$ is $b_1^*\sF_m$-invariant. Since $b_1^*\partial_m|_{U_A}$ generates $b^*\sF_m|_{U_A}$, without having to saturate, we have $a(E;\sF_m)=0$.
		\item \textsc{Chart $U_B$}, given by $x\mapsto uv, \ y\mapsto v, \ z\mapsto z, \ t\mapsto vs$. Then 
					$$b_1^*\partial_m|_{U_B}= (v^2-u^2)\partial_u+uv\partial_v-us\partial_s+v^ms^m\partial_z$$
			which is singular along the plane $u=v=0$. These singularities are non-lc.
		\item \textsc{Chart $U_C$}, given by $x\mapsto us, \ y\mapsto vs, \ z\mapsto z, \ t\mapsto s$. Then 
					$$b_1^*\partial_m|_{U_C}= v^3s^2\partial_u+u\partial_v+s^m\partial_z$$
			which is singular along the plane $u=s=0$. These singularities are non-lc.
	\end{itemize}
We see that the locus $S\subset Y$ of singularities of $b_1^*\sF_m$ is smooth, closed and irreducible of codimension $2$. As can be checked on chart $U_B$, that singular locus is not contained in the strict transform $b_*^{-1}(t=0)$.

Next let us blow-up $S$. We have two blow-up charts above $U_B$ (resp. above $U_C$), which we denote by $U_{BA}$ and $U_{BB}$ (resp. by $U_{CA}$ and $U_{CB}$):
	\begin{itemize}
		\item \textsc{Chart $U_{BA}$}, given by $u\mapsto \bar{u},\ v\mapsto \bar{u}\bar{v}, \ z\mapsto z, \ s\mapsto s$. Then 
				$$b^*\partial_m|_{U_{BA}}= \bar{u} \cdot \underbrace{\left[
				\bar{u}(\bar{v}^2-1)\partial_{\bar{u}}
				+\bar{v}(2-\bar{v}^2)\partial_{\bar{v}}
				-s\partial_s + \bar{u}^{m-1}\bar{v}^ms^m\partial_z\right]}_{\psi_{BA}}.$$
		As $\psi_{BA}(s)=-s$, by \autoref{lemma:p_power_of_derivation} we see that $\psi_{BA}^{[p]}=-\psi_{BA}$. Thus $b^*\sF_m$ is lc on the chart $U_{BA}$. As $\psi_{BA}(\bar{u})\in (\bar{u})$ we see that $F$ is $b^*\sF_m$-invariant. As $b^*\partial_m=\bar{u}\cdot \psi_{BA}$ we find that $a(F;\sF_m)=-1$.
		\item \textsc{Chart $U_{BB}$}, given by $u\mapsto \bar{u}\bar{v},\ v\mapsto \bar{v}, \ z\mapsto z, \ s\mapsto s$. Then 
				$$b^*\partial_m|_{U_{BB}}= \bar{v} \cdot \underbrace{\left[
				(1-2\bar{u}^2)\partial_{\bar{u}}
				+\bar{u}\bar{v}\partial_{\bar{v}}
				-\bar{u}s\partial_s + \bar{v}^{m-1}s^,\partial_z\right]}_{\psi_{BB}}.$$
		By \autoref{claim:multiplicative_derivation} we see that $\psi_{BB}(u)=2^{(5p-3)/2}\psi_{BB}(u)$. Thus by \autoref{prop:mult_is_lc} we see that $b^*\sF_m$ is lc on the chart $U_{BB}$.
		\item \textsc{Chart $U_{CA}$}, given by $u\mapsto \bar{u},\ v\mapsto v, \ z\mapsto z, \ s\mapsto \bar{u}\bar{s}$. Then 
				$$b^*\partial_m|_{U_{CA}}= \bar{u} \cdot \underbrace{\left[
				\bar{u}v^3\bar{s}^2\partial_{\bar{u}}
				+\partial_{v}
				+\bar{u}^{m-1}\bar{s}^m\partial_z - v^3\bar{s}^3\partial_{\bar{s}}\right]}_{\psi_{CA}}.$$
		As $\psi_{CA}$ generates a regular $1$-foliation, we see that $b^*\sF_m$ is regular on the chart $U_{CA}$.
		\item \textsc{Chart $U_{CB}$}, given by $u\mapsto \bar{u}\bar{s},\ v\mapsto v, \ z\mapsto z, \ s\mapsto \bar{s}$. Then 
				$$b^*\partial_m|_{U_{CB}}= \bar{s} \cdot \underbrace{\left[
				v^3\partial_{\bar{u}}
				+\bar{u}\partial_{v}
				+\bar{s}^{m-1}\partial_z\right]}_{\psi_{CB}}.$$
			Under the $\bA^1_t$-isomorphism $U_{CB}\cong \bA^3_{x,y,z}\times\bA^1_t$ given by $(\bar{u},\bar{v},z,\bar{s})\mapsto (x,y,z,t)$, we see that $\psi_{CB}$ corresponds to $\partial_{m-1}$, and thus $b^*\sF_m|_{U_{CB}}$ corresponds to $\sF_{m-1}$.
	\end{itemize}	
The above computations show that the non-lc locus of $b^*\sF_m$ is contained in $U_{CB}$. It is elementary (but tedious) to check that this locus does not intersect the other patches $U_{CA}$, $U_{BA}$ and $U_{BB}$. It remains to observe that the strict transforms of $E\subset Y$ and of $(t=0)\subset \bA^3\times\bA^1_t$ are disjoint from the chart $U_{CB}$. 
\end{proof}

\begin{theorem}\label{thm:loc_stable_families_II}
With the notations as above, assume that $p=3$ and let $Y=(\bA^3\times\bA^1)/\sF_1$. Then:
	\begin{enumerate}
		\item The projection $\bA^3\times \bA^1_t\to\bA^1_t$ factors through a flat morphism $Y\to \bA^1_t$ of relative dimension $3$;
		\item $Y_t$ is smooth for $t\neq 0$, while $Y_0$ is reduced and non-$S_2$;
		\item $Y\to \bA^1$ is locally stable, and $Y$ has canonical non-$S_3$ singularities along $Y_0$.
	\end{enumerate}
Moreover, if $C$ is a normal curve and $C\to \bA^1_t$ is a finite flat morphism, then $Y\times_{\bA^1_t}C$ is normal and $Y_C\to C$ is locally stable.
\end{theorem}
\begin{proof}
The first item is clear. For the remaining ones, as in the proof of \autoref{cor:canonical_non_S3} and \autoref{thm:stable_families}, we reduce to prove the following: if $Y'=(\bA^3\times\bA^1)/\sF_{p^r}$ then $Y'\to\bA^1_t$ is locally stable, has smooth fibers above $t\neq 0$ and a non-$S_2$ reduced central fiber $Y'_0$. The statements about the singularities of the fibers are proved as in \autoref{thm:loc_stable_families_I} (and are true whenever $p>2$). It remains to check that $Y'\to \bA^1$ is locally stable above $t=0$. 

As usual by now, we blow-up $W=\bA^3\times\bA^1_t$. By \autoref{lemma:blow-up_reduce_degree_II}, there is a sequence of smooth blow-ups
		$$b\colon X=X_{2p^r}\longrightarrow W=\bA^3\times\bA^1_t$$
with exceptional divisors $E_i,F_i\subset X$ for $i=1,\dots,p^r$ such that: $(X,\sum_i E_i+F_i)$ is log smooth, $b^*\sF_{p^r}$ has only lc singularities, the $E_i$'s and the $F_i$'s are $b^*\sF_{p^r}$-invariant, and
		$$a(E_i;\sF_{p^r})=-i+1, \quad a(F_i;\sF_{p^r})=-i\quad
		\forall i=1,\dots,p^r.$$
Using \autoref{lemma:blow-up_reduce_degree_II} and proceeding by induction, one computes
		$$a(E_i; W,W_0)=3i-2,\quad a(F_i;W,W_0)=3i-1.$$
Now let $Z=X/b^*\sF_{p^r}$ and denote by $E'_i\subset Z$ (resp. by $F'_i\subset Z$) the image of $E_i$ (resp. of $F_i$). Then by \autoref{lemma:blow-up_reduce_degree_II} and \autoref{thm:sing_qt} we have
		$$a(E_i';Y',Y'_0)=a(E_i;W,W_0)+(p-1)\cdot a(E_i;\sF_{p^r})=i(4-p)+p-3$$
and
		$$a(F_i';Y',Y'_0)=a(F_i;W,W_0)+(p-1)\cdot a(F_i;\sF_{p^r})=i(4-p)-1.$$
As $b^*\sF_{p^r}$ is lc, by \autoref{thm:sing_qt} it suffices to ensure that 
	$$\discrep\left(Z, -\sum_i a(E_i';Y',Y_0')E_i'-\sum_i a(F_i';Y',Y_0')F_i'\right)\geq -1$$
to obtain that $(Y',Y_0')$ is lc. This holds as soon as $i(4-p)\geq 0$ for every $i=1,\dots,p^r$, which happens if (and only if) $p=3$, so the proof is complete.
\end{proof}


\begin{remark}\label{rmk:compactification_is_difficult}
One can try to compactify the examples of \autoref{thm:loc_stable_families_II} using the method of \autoref{thm:stable_families}. This is quite delicate, as
the $1$-foliation $\psi=y^3\partial_x+x\partial_y$ acquires complicated singularities when we compactify. For example, if we regard $\bA^2_{x,y}$ has a standard chart of $\bP^2$, then on the chart with coordinates $u=1/x, v=y/x$ the $1$-foliation induced by $\psi$ is generated by $uv^3\partial_u+(v^4-u^2)\partial_v$. Similarly, if we regard $\bA^2_{x,y}$ has a standard chart of the Hirzebruch surface $F_n$, then on the chart with coordinates $u=1/y, v=1/(xy^n)$ the $1$-foliation induced by $\psi$ is generated by $u^{2n+5}\partial_u+(v^3+nu^{2n+4}v)\partial_v$.
\end{remark}

\appendix
\section{KSBA moduli stacks in positive characteristics}\label{appendix}
In this appendix we explore the consequences of \autoref{thm_intro:stability} for KSBA moduli theory. The point is to show that the most straightforward adaptations of the characteristic $0$ definitions do not yield satisfactory theories.

As we will see, the point is the unavoidable appearance of non-$S_2$ fibers as limits of stable families over punctured curves. To formulate this in a precise way we need to define which moduli stacks we are working with (or, equivalently, which families we allow in our theory), and we face two difficulties:
	\begin{enumerate}
		\item Our examples involve boundary divisors with small coefficients, and defining stable families of pairs with such coefficients over general bases is already extremely delicate in characteristic $0$;
		\item The algebraicity of the stacks we could write down is an open question (\footnote{
			On the other hand, any reasonable KSBA moduli stack will have a finite diagonal by \cite{Patakfalvi_Projectivity_moduli_space_of_surfaces_in_pos_char}, first paragraph of the proof of Theorem 9.7, and will be separated by \autoref{lemma:extension_iso} below.
		}), except in the surface case where it should follow from \cite{Hacon_Kovacs_Boundedness_slc_surfaces_general_type}.
	\end{enumerate}
We go around the first difficulty by using a minimalistic, under-determined, working definition of the KSBA moduli stacks. We will ignore the second difficulty by imposing algebraicity in our definition: it is natural to expect from the point of view of moduli theory (since we expect boundedness to hold), and on a technical level it allows us to state generic conditions for the objects parametrized by the stacks. In any case, (non-)algebraicity is irrelevant for the phenomenon that we will exhibit.


\medskip
We work over an algebraically closed field $k$ of characteristic $p>0$. First, we define one-parameter stable pairs and stable families of pairs, generalizing slightly \autoref{def:local_stability}.

\begin{definition}
A proper pair $(X,\Delta)$ over $k$ is \emph{stable} if it is semi-log canonical (slc) (\footnote{
	We refer to \cite{Kollar_Singularities_of_the_minimal_model_program} and to \cite{Posva_Gluing_for_surfaces_and_threefolds} for the definition of semi-log canonical singularities.
}) and if $K_X+\Delta$ is ample.
\end{definition}

\begin{definition}\label{def_app:stable_families}
Let $T$ be regular $k$-curve. Let $f\colon X\to T$ be a flat proper pure-dimensional morphism with geometrically reduced fibers. Let $\Delta$ be a $\bQ$-divisor on $X$ such that $f\colon (X,\Delta)\to T$ is a family of pairs \cite[Definition 2.2]{Kollar_Families_of_varieties_of_general_type}. We say that $f\colon (X,\Delta)\to T$ is a \emph{stable family} if:
	\begin{enumerate}
		\item $(X,X_t+\Delta)$ is slc for every closed point $t\in T$, and
		\item $K_X+\Delta$ is $f$-ample.
	\end{enumerate}
\end{definition}

\begin{lemma}[cf.\ {\cite[Lemma 9.4]{Patakfalvi_Projectivity_moduli_space_of_surfaces_in_pos_char}}]\label{lemma:extension_iso}
Let $T$ be an affine regular $k$-curve. Let $(X_i,\Delta_i)\to T$ $(i=1,2)$ be two stable families over $T$ whose total spaces $X_i$ are normal. Assume that for a fixed closed point $0\in T$, with complement $T^*=T\setminus \{0\}$, there is a $T^*$-isomorphism of pairs
		$$\phi^*\colon (X_1,\Delta_1)\times_TT^*\cong (X_2,\Delta_2)\times_T T^*.$$
Then $\phi^*$ extends to a $T$-isomorphism $\phi\colon (X_1,\Delta_1)\cong (X_2,\Delta_2)$.
\end{lemma}
\begin{proof}
If $Z$ is the normalization of the closure of the graph of $\phi^*$, with projections $\alpha_i\colon Z\to X_i$, let us write for each $i$
		$$\alpha_i^*(K_{X_i}+\Delta_i)=K_Z+\underbrace{\Gamma_{i,-}}_{\leq 0}+\underbrace{\Gamma_{i,+}}_{\geq 0}$$
where $\Gamma_{i,-}$ and $\Gamma_{i,+}$ have no components in common. We have
		$$X_i=\Proj_T \bigoplus_{m\geq 0}H^0(X_i, \lfloor m(K_{X_i}+\Delta_i\rfloor)
		= \Proj_T \bigoplus_{m\geq 0} H^0(Z, \lfloor m(K_Z+\Gamma_{i,+}\rfloor).$$
Now since $(X_i,\Delta_i+X_{i,0})$ is lc and $X_{i,0}$ is Cartier, if $E$ is an exceptional divisor over $X_i$ with $a(E;X_i,\Delta_i)<0$ then $E$ dominates $T$. In particular, every component of $\Gamma_{i,+}$ dominates $T$. By assumption on $\phi^*$ it follows that $\Gamma_{1,+}=\Gamma_{2,+}$ and therefore the above $\Proj$ description of the $X_i$ yields the extension $\phi\colon (X_1,\Delta_1)\cong (X_2,\Delta_2)$. 
\end{proof}

Now we want to give our working definition of KSBA moduli stacks. We will only specify their values on points and regular curves, using the two definitions above. But at least for psychological comfort, it is better to specify what kind of values our stacks take in general. We follow the first step of the approach of \cite{Kollar_Families_of_varieties_of_general_type} (see in particular Chapter 7 there): the objects of interest lie amongst \emph{relative effective Mumford divisors}. 

We refer to \cite[4.29, 4.78]{Kollar_Families_of_varieties_of_general_type} for the definition of a relative effective Mumford divisor $D$ supported on a morphism $X\to T$. To simplify the terminology, from now on a relative effective Mumford divisor will refer to the $2$-tuple $(X\to T, D)$. As explained in \cite[4.29]{Kollar_Families_of_varieties_of_general_type}, given any $h\colon T'\to T$ we can form the pullback $h_X^{[*]}D$ which is again a relative effective Mumford divisor supported on $X_{T'}\to T'$, and this pullback operation is functorial. Thus we can make the following definition:

\begin{definition}[Fibered category of relative Mumford divisors]
Fix a (possibly empty) vector of positive rational numbers $\bold{c}=(c_1,\dots,c_r)\in (\bQ_{>0})^{\oplus r}$. We define the category $\Mum_\bold{c}$ as follows:
	\begin{itemize}
		\item its objects are $2$-tuples $\sD=(f\colon X\to T,\sum_{i=1}^rc_iD_i)$ where $f$ is a morphism of $k$-schemes, and each $(f\colon X\to T,D_i)$ is a relative effective Mumford divisor supported on $f$;
		\item the class of arrows is generated by the following ones:
			\begin{enumerate}
				\item if $\sD=(X\to T,\sum_{i=1}^rc_iD_i)$ and $\sD'=(X'\to T,\sum_{i=1}^rc_iD_i')$ are two relative Mumford divisors over a common base $T$, then every $T$-isomorphism $\phi\colon X\cong X'$ sending $D_i$ to $D_i'$ for each $i$ defines an arrow $\phi\colon \sD\cong \sD'$;
				\item if $\sD$ is a relative Mumford divisor over $T$ and $h\colon T'\to T$ is a $k$-morphism, then there is a pullback morphism $h^{[*]}\sD\to\sD$.
			\end{enumerate}
	\end{itemize}
We have a forgetful functor $\Mum_\bold{c}\to \Sch_k$ sending $(X\to T,\sum_{i=1}^rc_iD_i)$ to $T$, the isomorphism $\phi$ to $\id_T$, and $h^{[*]}\sD\to\sD$ to $h$. By functoriality of pullbacks, this makes $\Mum_\bold{c}$ into a fibered category over $\Sch_k$.
\end{definition}

With these definitions in place, we can at last introduce our working definition of KSBA moduli stacks. 

\begin{terminology}
A \emph{stack over $k$} means a stack in groupoids over the big \'{e}tale site of $k$ in the sense of \cite[02ZH]{Stacks_Project}. \emph{Algebraic stacks} are understood in the sense of \cite[026N]{Stacks_Project}.
\end{terminology}

\begin{definition}[Potential KSBA moduli stacks]\label{def:KSBA_mod_stacks}
Let $\sM$ be an algebraic stack over $k$. We say that $\sM$ is a \emph{potential KSBA moduli stack} if there exist $n\in \bN$, $v\in \bQ_{>0}$ and $\bold{c}=(c_1,\dots,c_r)\in (\bQ_{>0})^{\oplus r}$ with the following properties:
	\begin{enumerate}
		\item There exists a fully faithful functor $\sM\hookrightarrow \Mum_\bold{c}$ of fibered categories over $\Sch_k$;
		\item if $K$ is an algebraically closed field extension of $k$, then $\sM(K)$ is the groupoid of all stable pairs $(X\to \Spec(K),\Delta=\sum_{i=1}^rc_i\Delta_i)\in \Mum_\bold{c}(K)$ with $\dim X=n$ and $(K_X+\Delta)^n=v$;
		\item if $T$ is a regular (germ of) $k$-curve then $\sM(T)$ is the set of $(X\to T,\Delta=\sum_{i=1}^rc_i\Delta_i)\in \Mum_\bold{c}(T)$ such that $(X,\Delta)\to T$ is a stable family.
	\end{enumerate}
We refer to the array $(n,v,\bold{c})$ as the \emph{numerical constants} of $\sM$.
\end{definition}

\begin{remark}\label{rmk:stable_families_with_nice_fibers}
Let us stress that if $\sM$ is a potential KSBA moduli stack and $T$ a regular $k$-curve, then the families in $\sM(T)$ are stable families of pairs \emph{whose underlying fibers are demi-normal (in particular reduced and $S_2$)}. This is forced by the way we compute pullbacks in $\Mum_\bold{c}$ and by the prescription of the values of $\sM$ on points.
\end{remark}

We introduce two variants of the above definition.

\begin{variant}[Potential KSBA-CM moduli stacks]\label{def:KSBA-CM_mod_stacks}
Let $\sM$ be a potential KSBA moduli stack over $k$. We say that $\sM$ is a \emph{potential KSBA-CM moduli stack} if there exists a dense open algebraic sub-stack $\sM^\text{CM}\subset \sM$ such that, whenever $(X,\Delta)\in \sM^\text{CM}(K)$ for an algebraically closed field $K$, the variety $X$ is Cohen--Macaulay.
\end{variant}

\begin{remark}
Because of \cite[12.2.1]{EGA_IV.3}, we think of a potential KSBA-CM moduli stack as an open sub-stack of a potential KSBA moduli stack. In characteristic $0$ we have the following striking picture: if $\fM=\fM^\text{KSBA}_{\bC,n,v,\bold{c}}$ is the KSBA moduli stack over $\bC$ with numerical constants $(n,v,\bold{c})$, defined as in \cite[\S 8.2]{Kollar_Families_of_varieties_of_general_type}, then we have a decomposition into \emph{open and closed} sub-stacks
		$$\fM=\bigsqcup_{i=2}^n\fM(i)$$
where $\fM(i)$ parametrizes pairs whose underlying varieties are $S_i$ but not $S_{i+1}$. This follows from \cite{Kollar_Kovacs_Deformations_of_lc_and_F-pure_sing} (see also \cite[Corollary 1.3]{Kollar_Kovacs_Lc_is_Du_Bois}). In particular $\fM^\text{CM}=\fM(n)$ is a connected component of $\fM$.
\end{remark}

\begin{variant}[Potential KSBA-$F$-injective moduli stacks]\label{def:KSBA-CM_mod_stacks}
Let $\sM$ be a potential KSBA moduli stack over $k$. We say that $\sM$ is a \emph{potential KSBA-$F$-injective moduli stack} if there exists a dense open algebraic sub-stack $\sM^{F\text{-inj}}\subset \sM$ such that, whenever $(X,\Delta)\in \sM^{F\text{-inj}}(K)$ for an algebraically closed field $K$, the variety $X$ is $F$-injective.
\end{variant}

$F$-injectivity is often described as a positive characteristic analogue of the du Bois condition in characteristic $0$: see for example \cite{Schwede_F-inj_are_Du_Bois} and \cite{Kovacs_Schwede_Hodge_theory_meets_MMP}. 

\begin{remark}[Generic singularities in \cite{Kollar_Families_of_3folds_in_pos_char}]\label{rmk:sing_Kollar_example}
We briefly comment on the singularities of a general fiber $Y^c_1$ of the examples of \cite{Kollar_Families_of_3folds_in_pos_char}, using the notations of that paper. As noted in \cite[Lemma 17]{Kollar_Families_of_3folds_in_pos_char}, $Y_1^c$ is not CM. It is easily seen from the construction that $Y_1^c$ is not lc: a resolution is given by $\tau_Y\colon Y_1\to Y_1^c$, and the exceptional locus is the divisor $Z_0$ which is a log canonical place. We check whether $Y_1^c$ is $F$-injective along $\tau_Y(Z_0)$. Recall that $Z_0\to \tau_Y(Z_0)$ is the elliptic fibration $\tau\colon S_1\to \bP^1$ \cite[Proof of Lemma 17]{Kollar_Families_of_3folds_in_pos_char}. Arguing as in \cite[Lemma 3.10, Proposition 3.12]{Arvidsson_Bernasconi_Patakfalvi_Properness_moduli_surfaces}, we see that for any $y\in \tau_Y(Z_0)$ we have 
	$$H^2_y(Y_1^c,\sO_{Y^c_1})=H^0_y(Y_1^c, R^1\tau_{Y,*}\sO_{Y_1})
	=H^0_y(\bP^1, R^1\tau_*\sO_{S_1}).$$
The fibration $\tau$ has one multiple fiber $pD_1$, and the other fibers are smooth elliptic curves \cite[\S 9]{Kollar_Families_of_3folds_in_pos_char}. The torsion of $R^1\tau_*\sO_{S_1}$ is concentrated at the image $\mathfrak{y}$ of $pD_1$, and we have $\dim H^0_{\mathfrak{y}}(\bP^1, R^1\tau_*\sO_{S_1})=1$ \cite[Proof of Lemma 17]{Kollar_Families_of_3folds_in_pos_char}. By cohomology and base-change, this torsion module embeds into $H^1(pD_1,\sO_{pD_1})$, and to check $F$-injectivity it remains to understand the action of Frobenius on this cohomology group. By \cite[Proof of Lemma 17]{Kollar_Families_of_3folds_in_pos_char} we have $H^1(pD_1,\sO_{pD_1})=H^1(E,F_{p-1}\oplus \sO_E)$ where $E$ is an elliptic curve and $F_{p-1}$ is the unique unipotent indecomposable bundle of rank $p-1$ on $E$. By induction on the rank of such bundles, we see that the action of Frobenius on $H^1(E,F_{p-1}\oplus \sO_E)$ is bijective if and only if $E$ is ordinary. In particular, we may arrange $Y_1^c$ to be $F$-injective.
\end{remark}



With all these preparations, we can finally formulate:

\begin{theorem}\label{thm:KSBA_mod_stacks}
Let $n\geq\max\{p,3\}$. Then there is a dense subset $I$ of $(0;+\infty)$ such that for every $v\in I$ there exists $\bold{c}=\bold{c}(v)\neq \emptyset$ with the following property: there is no \emph{proper} potential KSBA, nor KSBA-CM/$F$-injective, moduli stack over $k$ with numerical constants $(n,v,\bold{c})$.
\end{theorem}
\begin{proof}
Let $T=\bA^1_t$ and consider some family $f\colon (\sY,\sB)\to T$ constructed in \autoref{thm:stable_families} of relative dimension $n$. There are many choices for $\sB$, and we exploit this later. For now let $T^*=T\setminus\{0\}$ and $f^*\colon (\sY^*,\sB^*)\to T^*$ denote the family restricted over the punctured curve. A geometric fiber of $f^*$ has only $\mu_p$-quotient singularities: in particular it is Cohen--Macaulay and $F$-injective \cite[Theorem 1]{Posva_Singularities_quotients_by_1_foliations}. So if $\sM$ is a potential KSBA (or KSBA-CM/$F$-injective) moduli stack whose numerical constants match those of $(\sY^*,\sB^*)\to T^*$, we must have $[(\sY^*,\sB^*)\to T^*]\in \sM(T^*)$. Restricting over the generic point of $T^*$, we obtain $[f]\colon \Spec(k(t))\to \sM$. I claim that there is no finite extension of DVRs $k[t]_{(t)}\hookrightarrow R$ such that we have a commutative diagram over $k$
	$$\begin{tikzcd}
	\Spec(\Frac(R))\arrow[r, "\iota"]\arrow[d, hook] & \Spec(k(t)) \arrow[r, "{[}f{]}"] & \sM \\
	\Spec(R) \arrow[rru, "{[}g{]}" below right]
	\end{tikzcd}$$
where $\iota$ is induced by the DVRs extension. This implies non-properness of $\sM$ by \cite[0CLZ]{Stacks_Project}.

We proceed by contradiction. Assume that there exists such an extension of DVRs and morphism $[g]\colon \Spec(R)\to \sM$. This implies that there exists a stable family $g\colon (\sY',\sB')\to \Spec(R)$, belonging to $\sM(R)$, whose generic fiber is the base-change of the generic fiber of $(\sY,\sB)\to T$ along $\iota$. The scheme $\sY'$ is also normal, since its generic fiber is normal and its special fiber is reduced. Moreover, by \autoref{thm:stable_families} the base-change $\sY_R$ is normal. Therefore by \autoref{lemma:extension_iso} we actually have an $R$-isomorphism $(\sY_R,\sB_R)\cong (\sY',\sB')$. So $(\sY_R,\sB_R)\in \sM(R)$: but the central fiber of $\sY_R\to \Spec(R)$ is not $S_2$ by \autoref{thm:stable_families}, and we obtain a contradiction with the definition of the elements of $\sM(R)$ (see \autoref{rmk:stable_families_with_nice_fibers}).

It remains to prove that as we change the boundary $\sB$, the set of volumes $(K_{\sY_t}+\sB_t)^n$, for $t$ an arbitrary point in $T^*$, is dense in the interval $(0;+\infty)$. Recall that the support of $\sB_t$ is the image through a degree $p$ quotient map
		$$\bP^{n-1}\times E\longrightarrow \sY_t$$
of a divisor $H_0$. By \autoref{rmk:volume} it suffices to consider the case where $H_0$ belongs to the $\bQ$-linear systems 
	$$\frac{1}{s}\left|(\sO_{\bP^{n-1}}(n+p)\boxtimes\sO_E(e))^{\otimes r}\right|$$
where $r\geq 3$ and $s\geq 1$ are integers. From \autoref{eqn:pullback_formula} and \autoref{lemma:global_foliation} it follows that
		\begin{eqnarray*}
		(K_{\sY_t}+\sB_t)^n&=&p\cdot (K_{\bP^{n-1}\times E}+(p-1)K_{\sG_t}+H_0)^n \\
		&=& p\cdot \left[\sO_{\bP^{n-1}}\left(
		p-n-1+\frac{r}{s}(n+p)
		\right)\boxtimes \sO_E\left(\frac{r}{s}e\right)\right]^n \\
		&=& np\cdot \frac{r}{s}\cdot \left(
		p-n-1+\frac{r}{s}(n+p)
		\right)^n.
		\end{eqnarray*}
We consider this expression as a function $\nu$ in $r/s\in \bR_{>0}$. Then $\nu$ is continuous, increases to infinity with $r/s$, and is bounded below by
		$$\lim_{\frac{r}{s}\to 0^+} np\cdot \frac{r}{s}\cdot \left(
		p-n-1+\frac{r}{s}(n+p)
		\right)^n=0.$$
By continuity the image $I=\nu(\bQ_{>0})$ is dense in $(0;+\infty)$. By construction, for every $v=\nu(r/s)\in I$ there is a family $(\sY^*,\sB^*)\to T^*$ constructed by \autoref{thm:stable_families} which belongs to any potential KSBA or KSBA-CM moduli stacks with numerical constants $(n,v,\bold{c}\neq \emptyset)$ (\footnote{
	Here $\bold{c}$ depends on the choice of $H_0$: since we can take it to have coefficients $1/s$, one possible value of $\bold{c}$ is the 1-vector $(1/sp)$.
	}). This completes the proof.
\end{proof}

\bibliographystyle{alpha}
\bibliography{Bibliography}

\end{document}